\documentclass[a4paper,11pt]{amsart}

\usepackage{graphicx,color}
\usepackage{amsmath,amssymb}
\usepackage{amsthm}
\usepackage{enumitem}
\usepackage{mathtools}
\usepackage[all,cmtip,2cell]{xy}
\usepackage{array}
\usepackage{here}
\usepackage{cite}

\theoremstyle{plain}
\newtheorem{thm}{Theorem}[section]
\newtheorem{prop}[thm]{Proposition}
\newtheorem{lem}[thm]{Lemma}
\newtheorem{cor}[thm]{Corollary}
\newtheorem*{thm*}{Theorem}
\newtheorem{prob}{Problem}[section]

\theoremstyle{definition}
\newtheorem{defi}[thm]{Definition}

\newtheorem{nota}{Notation}
\newtheorem*{nota*}{Notation}
\newtheorem*{conv}{Conventions}

\theoremstyle{remark}

\newtheorem{ex}[thm]{Example}

\newtheorem*{claim*}{Claim}
\newtheorem*{claimproof*}{Proof of Claim}
\numberwithin{equation}{subsection}

\newcommand{\Z}{\mathbb{Z}}
\newcommand{\Q}{\mathbb{Q}}

\newcommand{\C}{\mathbb{C}}
\newcommand{\A}{\mathbb{A}}
\renewcommand{\P}{\mathbb{P}}
\newcommand{\F}{\mathbb{F}}

\newcommand{\vp}{\varphi}
\newcommand{\mc}{\mathcal}

\newcommand{\tb}{\textbf}

\newcommand{\ol}{\overline}
\newcommand{\wt}{\widetilde}

\newcommand{\ra}{\rightarrow}
\newcommand{\la}{\leftarrow}

\newcommand{\dra}{\dashrightarrow}
\newcommand{\xra}{\xrightarrow}
\newcommand{\xla}{\xleftarrow}

\DeclareMathOperator{\Sym}{\mathrm{Sym}}

\DeclareMathOperator{\PProj}{\tb{Proj}}
\DeclareMathOperator{\Supp}{Supp}

\DeclareMathOperator{\NE}{NE}

\DeclareMathOperator{\Pic}{\mathrm{Pic}}
\DeclareMathOperator{\eu}{\chi_{\mathrm{top}}}

\DeclareMathOperator{\Bl}{\mathrm{Bl}}

\DeclareMathOperator{\Sing}{\mathrm{Sing}}

\title[Compactifications into quadric fibrations]{On compactifications of affine homology 3-cells into quadric fibrations}
\author[M. Nagaoka]{Masaru Nagaoka}
\address{Graduate School of Mathematical Sciences\\The University of Tokyo\\3-8-1 Komaba\\Meguro-ku, Tokyo 153-8914, Japan}
\email{nagaoka@ms.u-tokyo.ac.jp}
\begin{document}
\maketitle
\begin{abstract}
In this paper we deal with compactifications of affine homology $3$-cells into quadric fibrations such that the boundary divisors contain fibers.
We show that all such affine homology $3$-cells are isomorphic to the affine $3$-space $\A^{3}$.
Moreover, we show that all such compactifications can be connected by explicit elementary links preserving $\A^{3}$ to the projective $3$-space $\P^{3}$.
\end{abstract}
\tableofcontents
\setcounter{section}{0}

\section{Introduction.}\label{sec:intro}

Throughout the paper we work over the field of complex numbers $\C$.

In \cite{Hir54}, F.\,Hirzebruch raised the problem to classify all the compactifications of the affine space $\A^{n}$ into smooth complete complex manifolds $X$ with the second Betti number $B_{2}(X)=1$. 
A trivial example of such compactifications is the pair of the projective space and its hyperplane $(\P^{n}, \P^{n-1})$.
This is the unique compactification of $\A^{n}$ when $n=1$, as is easy to check, and when $n=2$, which is proved by R.\,Remmert and T.\,Van de Ven \cite{RV}.
Non-trivial examples appear in dimension three, 
and when $X$ is projective, the complete classification of compactifications was achieved by M.\,Furushima, N.\,Nakayama, Th.\,Peternell and M.\,Schneider \cite{Fur86, Fur90, Fur93b, Fur93a, F-N89b, F-N89a, Pet89, P-S88}.
When $B_{2}=2$, non-trivial compactifications appear even in dimension two, which was studied by S.\,Mori \cite{Mor2}.

In this paper we are interested in compactifications of \textit{affine homology $n$-cells} into smooth projective $n$-fold.
By an affine homology $n$-cell we mean a smooth affine $n$-fold $U$ such that $H_{i}(U, \Z)=0$ for $i>0$. 
Furushima \cite{Fur00} pointed out that we can apply the same arguments in \cite{Fur86, Fur90, Fur93b, Fur93a, F-N89b, F-N89a, Pet89, P-S88} to classify all the compactifications of affine homology 3-cells into smooth Fano $3$-folds with $B_{2}=1$.
As a result, when a homology $3$-cell $U$ is embedded into a smooth Fano $3$-fold with $B_{2}=1$, it holds that $U \cong \A^{3}$.
On the basis of this result, T.\,Kishimoto \cite{Kis} and the author \cite{Nag} investigated compactifications of contractible affine $3$-folds into smooth Fano $3$-folds with $B_{2}=2$, where we note that a contractible affine $n$-fold is an affine homology $n$-cell with trivial fundamental group.

In this paper, we consider the corresponding problem in the case where the image of an extremal contraction is a curve. Namely:

\begin{prob}\label{prob:main}
Let $f \colon X \to C$ be an extremal contraction of relative Picard number one from a smooth projective $n$-fold $X$ to a smooth projective curve $C$. Let $U \subset X$ be an open subscheme.
\begin{enumerate}

\item[\textup{(1)}] If $U$ is an affine homology $n$-cell, then is it isomorphic to $\A^{n}$?

\item[\textup{(2)}] If $U$ is isomorphic to $\A^{n}$, then can we construct an explicit birational map from $X$ to a compactification of $\A^{n}$ with Picard number one preserving $U \cong \A^{n}$?

\end{enumerate}
\end{prob}

In this problem, we set not only $\P^{n}$ but also all compactification of $\A^{n}$ with Picard number one as the target of birational maps preserving $\A^{n}$.
It is because there is a copy of $\A^{3}$ in the quintic del Pezzo $3$-fold with the boundary divisor normal which we can regard naturally as an affine modification (for the detail, see \cite{K-Z}) of an another copy of $\A^{3}$ in $\P^{3}$ via the birational map constructed in \cite{Fur00}.

We note that even when $n=2$, Problem \ref{prob:main} (1) have a negative answer in general.
In fact, T.\ tom Dieck and T.\ Petrie \cite{tDP} showed that there are infinitely many contractible affine surfaces of logarithmic Kodaira dimension one in the blow-up of $\P^{2}$ at a point.
However, if we assume the following condition, the problem have an affirmative answer in the case where $n=2$.

\begin{defi}\label{defi:compatible}
Let $f \colon X \to C$ and $U$ be as in Problem \ref{prob:main}. 
Let $D \coloneqq X \setminus U$ be the boundary divisor.
We say that $(X, D, f)$ is \textit{a compactification of $U$ compatible with $f$} if $D$ contains a $f$-fiber.
When $D_{f} \subset D$ is a $f$-fiber and $D_{h} \subset D$ is the other components,
we also call $(X, D_{h}, D_{f})$ \textit{a compactification of $U$ compatible with $f$}.
\end{defi}

By \cite[Proposition 2.1]{van} and the Poincare duality, $D_{h}$ in the setting of Definition \ref{defi:compatible} is a prime divisor.
Suppose that $(X, D_{h}, D_{f})$ is a compactification of homology $2$-cell $U$ compatible with $\P^{1}$-bundle. 
By \cite[Corollary 1.20]{Fuj}, it holds that $D_{h}$ is a $f$-section.
Hence we have $U \cong \A^{2}$ since $f|_{U}$ is an $\A^{1}$-bundle over $\A^{1}$.

Problem \ref{prob:main} (2) was solved when $n=2$ by Mori \cite{Mor2}. 
He introduced three kinds of explicit birational transformations preserving $\A^{2}$ between Hirzebruch surfaces, which are called J-, R-, and L-transform. 
He solved the problem as in the following theorem:

\begin{thm}
Let $f \colon X \to \P^{1}$ be a $\P^{1}$-bundle and $D$ a reduced effective divisor on $X$ such that $X \setminus D \cong \A^{2}$. 
\begin{enumerate}

\item[\textup{(1)}] There exists a compactification $(X_{1}, D_{1}, f_{1})$ of $\A^{2}$ compatible with a $\P^{1}$-bundle $f _{1}\colon X_{1} \to \P^{1}$ and a birational map $g_{1} \colon X \dra X_{1}$ preserving $\A^{2}$ which is a finite composition of J-, R-, and L-transforms. 

\item[\textup{(2)}] Let $X_{2} \coloneqq \F_{1}$ be a Hirzebruch surface of degree $1$ with the $\P^{1}$-bundle structure $f_{2}$. 
Let $D_{2}$ be the union of an $f_{2}$-fiber and the minimal section. 
Then there exists a birational map $g_{2} \colon X_{1} \dra X_{2}$ preserving $\A^{2}$ which is a finite composition of elementary transformations of $\P^{1}$-bundles.

\end{enumerate}
Summarizing, we have the following diagram of birational maps preserving $X \setminus D \cong \A^{2}$:
\begin{equation}
\xymatrix@C=15pt@R=15pt{
(X, D) \ar@{.>}[r]^-{g_{1}} \ar[d]_{f}
&(X_{1}, D_{1}) \ar@{.>}[r]^-{g_{2}} \ar[d]_{f_{1}}
&(X_{2}, D_{2}) \ar[r]^-{g_{3}} \ar[d]_{f_{2}}
&(\P^{2}, \P^{1}) 
\\ \P^{1} \ar@{=}[r]
&\P^{1} \ar@{=}[r]
&\P^{1},
& 
}
\end{equation}
where $g_{3} \colon X_{2} \to \P^{2}$ is the blow-down of the minimal section.
\end{thm}

In this paper we consider Problem \ref{prob:main} when $n=3$ and $(X, D_{h}, D_{f})$ is compatible with $f$. 
In this case, $f$ is a del Pezzo fibration.
When $f$ is a $\P^{2}$-bundle, then the problem is easy by the same reason as when $n=2$ (see \S \ref{sec:mainA}).
However, if the degree of a general $f$-fiber is smaller than $9$, then the problem is not obvious since a general ($f|_{U}$)-fiber often differs from $\A^{2}$.

The main purpose of this paper is to give a solution to Problem \ref{prob:main} for compactifications compatible with a quadric fibration, i.e.\,when the degree of a general fiber is $8$.
Our main result consists of three theorems. 
One is the following theorem, which is the solution to Problem \ref{prob:main} (1).

\begin{thm}\label{thm:main0}
Let $q \colon Q \ra C$ be a quadric fibration, $D_{h}$ a reduced effective divisor on $Q$, and $D_{f}$ a $q$-fiber.
Then the following are equivalent.
\begin{enumerate}
\item[\textup{(1)}] The complement $Q \setminus (D_{h} \cup D_{f})$ is an affine homology $3$-cell. 
\item[\textup{(2)}] It holds that $C \cong \P^{1}$ and $Q \setminus (D_{h} \cup D_{f}) \cong \A^{3}$.
\end{enumerate}
\end{thm}

The others are Theorem \ref{thm:maindiag-A} and \ref{thm:maindiag-B}, which give a solution to Problem \ref{prob:main} (2).
Before stating the theorems, we introduce some examples of compactifications of $\A^{3}$ compatible with del Pezzo fibrations and explicit birational maps preserving $\A^{3}$ from them to $\P^{3}$.

\begin{ex}\label{ex:0-A}
Let $g_{3} \colon P' \to \P^{3}$ be the blow-up along a line and $D_{h,2}$ the exceptional divisor.
Then the linear system $|g_{3}^{*}\mc{O}_{\P^{3}}(1)-D_{h,2}|$ defines a $\P^{2}$-bundle structure $p' \colon P' \to \P^{1}$.
Let $D_{f,2}$ be a $p'$-fiber.
Then $(P', D_{h,2}, D_{f,2})$ is a compactification of $\A^{3}$ compatible with $p'$ because $P' \setminus (D_{h, 2} \cup D_{f,2}) \cong \P^{3} \setminus g_{3*}D_{f,2} \cong \A^{3}$.
Hence $g_{3} \colon P' \to \P^{3}$ is a birational map preserving $\A^{3}$.
\end{ex}
\begin{ex}\label{ex:0-B}
 Let $h_{2} \colon Q' \to \Q^{3}$ be the blow-up of the smooth quadric $\Q^3 \subset \P^{4}$ along a smooth conic and $D_{h}'$ the exceptional divisor.
Then 
the linear system $|h_{2}^{*}\mc{O}_{\Q^{3}}(1)-D_{h}'|$ defines a quadric fibration structure $q' \colon Q' \to \P^{1}$.
Let $D_{f}'$ be a singular $q'$-fiber, which is isomorphic to the quadric cone $\Q^{2}_{0} \subset \P^{3}$.
Then $h_{2}$ induces an isomorphism $Q' \setminus (D_{h}' \cup D_{f}') \cong \Q^{3} \setminus \Q^{2}_{0}$.
Let $h_{3} \colon \Q^{3} \dra \P^{3}$ be the projection from the vertex of $\Q^{2}_{0}$.
Then, by the discussion in \cite[pp.117--119]{Fur00}, $h_{3}$ induces an isomorphism $\Q^{3} \setminus \Q^{2}_{0} \cong \P^{3} \setminus \P^{2} \cong \A^{3}$.
Hence $(Q', D_{h}', D_{f}')$ is a compactification of $\A^{3}$ compatible with $q'$ and $h_{3} \circ h_{2} \colon Q' \dra \P^{3}$ is a birational map preserving $\A^{3}$.
\end{ex}

To solve Problem \ref{prob:main} (2) for all the compactifications of $\A^{3}$ compatible with quadric fibrations, we will use the technique of elementary links.

\begin{defi}\label{defi:elem}
Let $X$ be a smooth $3$-fold and $\sigma \colon X \ra C$ be an extremal contraction of relative Picard number one. 
Let $r \subset X$ be a smooth curve (or $x \in X$ be a point).
Denote by $\vp \colon \wt X \ra X$ the blow-up of $X$ with center $r$ (resp.\ with center $x$). 
We assume that $-K_{\wt X}$ is $(\sigma \circ \vp)$-ample.
Then there exists the unique contraction $\psi \colon \wt X \ra Y$ of the other $K_{\wt X}$-negative ray in $\ol{\NE}(\wt X/C)$.
Let $\tau \colon Y \ra C$ be the induced morphism.
\begin{equation}\label{diag:elem}
\xymatrix@!C=15pt@R=15pt{
&\wt{X} \ar[dl]_-{\vp} \ar[dr]^-{\psi}
&
\\X \ar[d]_-{\sigma}
&
&Y \ar[d]^-{\tau}
\\C \ar@{=}[rr]
& 
&C.
}
\end{equation}
When $\psi$ is birational, we call the diagram (\ref{diag:elem}) \textit{the elementary link with center along} $r$ (resp.\,\textit{at} $x$). 
In this paper, the pushforward of the \nolinebreak$\vp$-exceptional divisor by $\psi$ is called \textit{the exceptional divisor of the elementary link}.
We write it $X \la \wt X \ra Y$ or $X \dra Y$ for short when the base variety is obvious.
\end{defi} 

We note that this is a particular case of elementary links of type I\hspace{-.1em}I in \cite[Definition 3.4]{Cor} and that the exceptional divisor of the elementary link is actually a divisor by \cite[Proposition 3.5]{Cor}.

With the above notation, the other main theorems are stated as follows.
\begin{thm}\label{thm:maindiag-A}
Let $(Q, D_{h}, D_{f})$ be a compactification of $\A^{3}$ compatible with a quadric fibration $q \colon Q \ra \P^{1}$. Suppose that $D_{h}$ is non-normal. 
\begin{enumerate}
\item[\textup{(1)}] Let $g_{1} \colon Q \dra P$ be the elementary link with center along the singular locus of $D_{h}$, which is a $q$-section. 
Let $D_{f, 1}$ be the strict transform of $D_{f}$ in $P$ and $D_{h,1}$ the exceptional divisor of the elementary link.
Then $P$ has a $\P^{2}$-bundle structure $p$ over $\P^{1}$ and
 $(P, D_{h,1}, D_{f,1})$ is a compactification of $\A^{3}$ compatible with $p$.
\item[\textup{(2)}] We follow the notation of Example \ref{ex:0-A}.
Regard $p(D_{f,1})$ and $p'(D_{f,2})$ as $\infty \in \P^{1}$.
Then there is the composition $g_{2} \colon P \dra P'$ of elementary links with center along linear subspaces in the fibers at $\infty$ such that $D_{h,2}$ is the strict transform of $D_{h,1}$ in $P'$.
\end{enumerate}
Summarizing, we have the following diagram of rational maps preserving $Q \setminus (D_{h} \cup D_{f}) \cong \A^{3}$:
\begin{equation}
\xymatrix@C=15pt@R=15pt{
(Q, D_{h}, D_{f}) \ar@{.>}[r]^-{g_{1}} \ar[d]_{q}
&(P, D_{h, 1}, D_{f, 1}) \ar@{.>}[r]^-{g_{2}} \ar[d]_{p}
&(P', D_{h, 2}, D_{f, 2}) \ar[r]^-{g_{3}} \ar[d]_{p'}
&(\P^{3}, H) 
\\ \P^{1} \ar@{=}[r]
&\P^{1} \ar@{=}[r]
&\P^{1},
& 
}
\end{equation}
\nolinebreak where $H \coloneqq g_{3*}D_{f,2}$.
\end{thm}

\begin{thm}\label{thm:maindiag-B}
Let $(Q, D_{h}, D_{f})$ be a compactification of $\A^{3}$ compatible with a quadric fibration $q \colon Q \ra \P^{1}$. 
Suppose that $D_{h}$ is normal. 
We follow the notation of Example \ref{ex:0-B}.
Regard $q(D_{f})$ and $q'(D_{f}')$ as $\infty \in \P^{1}$.
Then there is the composition $h_{1} \colon Q \dra Q'$ of elementary links with center along rulings in the fibers at $\infty$ such that $D'_{h}$ is the strict transform of $D_{h}$ in $Q'$.
In particular, we have the following diagram of rational maps preserving $Q \setminus (D_{h} \cup D_{f}) \cong \A^{3}$:
\begin{equation}
\xymatrix@C=15pt@R=15pt{
(Q, D_{h}, D_{f}) \ar@{.>}[r]^-{h_{1}} \ar[d]_{q}
&(Q', D'_{h}, D'_{f}) \ar[r]^-{h_{2}} \ar[d]_{q'}
&(\Q^{3}, \Q^{2}_{0}) \ar@{.>}[r]^-{h_{3}}
&(\P^{3}, H)
\\ \P^{1} \ar@{=}[r]
&\P^{1},
& 
&
}
\end{equation}
where we regard $h_{2*}D_{f}'$ as $\Q^{2}_{0}$ and $H \coloneqq \P^{3} \setminus h_{3}(\Q^{3} \setminus \Q^{2}_{0})$.
\end{thm}

In Example \ref{ex:5}, we construct a compactification of an affine homology $3$-cell into a quadric fibration, which gives a negative answer to Problem \ref{prob:main} (1) in the case where $n=3$ without the assumption on the compatibility.
In the forthcoming paper \cite{Nag2}, we deal with Problem \ref{prob:main} for compactifications compatible with del Pezzo fibrations when the degree of a general fiber is less than $8$.
Problem \ref{prob:main} (2) for general compactifications into del Pezzo fibrations is at present far from being solved.

\vspace{11pt}

%\begin{ootp}
This article is structured as follows.

In \S \ref{sec:elem}, we recall some facts on elementary links.

%In \S \ref{sec:ex}, we construct two example of compactifications $(Q, D_{h}, D_{f})$ of $\A^{3}$ compatible with quadric fibrations from the standard compactification of $\A^{3}$ into $\Q^{3}$. 
%One of them is the same as $(Q', D_{h}', D_{f}')$ as in Theorem \ref{thm:maindiag-B} and the other has the non-normal $D_{h}$.

In \S \ref{sec:top}, we determine the Hodge diamonds of del Pezzo fibrations containing affine homology $3$-cells.
We also show that the base curve must be $\P^{1}$.

In \S \ref{sec:equiv}, we give precise statement of Theorem \ref{thm:main0} as in Theorem \ref{thm:main1} and prove it by using elementary links from quadric fibrations to $\P^{2}$-bundles.

In \S \ref{sec:ex}, we construct several examples of compactifications of $\A^{3}$ compatible with quadric fibrations as applications of Theorem \ref{thm:main1}.
We note that these examples are erroneously omitted from \cite[Table 1]{Kis} or \cite[Table 1]{Mul}.
We also construct an example of compactifications of affine homology $3$-cells.
This gives a negative answer to Problem \ref{prob:main} (1) in the case where $n=3$ and the compactification is not compatible with the extremal contraction.

In \S \ref{sec:mainA}, we give a solution to Problem \ref{prob:main} for compactifications compatible with $\P^{2}$-bundles.
Theorem \ref{thm:maindiag-A} follows as a corollary.

In the rest of paper, we prove Theorem \ref{thm:maindiag-B} as follows.
Let $(Q, D_{h}, D_{f})$ be a compactification of $\A^{3}$ compatible with a quadric fibration such that $D_{h}$ is normal.

First, in \S \ref{sec:sing}, we determine the singularities of $D_{h}$ and $D_{f}|_{D_{h}}$.
We also assign a non-negative integer to them, which we call the type of $(Q, D_{h}, D_{f})$.
By definition, $(Q, D_{h}, D_{f})$ is of type $0$ if and only if $D_{h}$ is a Hirzebruch surface.

Next, in \S \ref{sec:decrease}, we suppose that $(Q, D_{h}, D_{f})$ is of type $m>0$. We construct a birational map preserving $\A^{3}$ from $(Q, D_{h}, D_{f})$ to another compactification of type $(m-1)$ via elementary links between quadric fibrations.
Composing such maps, we get a birational map from $(Q, D_{h}, D_{f})$ to a compactification of $\A^{3}$ of type $0$. 
Hence we reduce to proving Theorem \ref{thm:maindiag-B} when $(Q, D_{h}, D_{f})$ is of type $0$ i.e.\,when $D_{h}$ is a Hirzebruch surface. 

Finally, in \S \ref{sec:desiredseq2}, we suppose that $D_{h}$ is a Hirzebruch surface of degree $d \in \Z_{\geq 0}$.
When $d>0$, we give a birational map preserving $\A^{3}$ from $(Q, D_{h}, D_{f})$ to another compactification $(Q', D_{h}', D_{f}')$ of $\A^{3}$ of type $0$ such that $D_{h}'$ is a Hirzebruch surface of degree $(d-1)$.
When $d=0$, we show that $(Q, D_{h}, D_{f})$ is actually the same as $(Q', D_{h}', D_{f}')$ as in Example \ref{ex:0-B}.
We have thus proved Theorem \ref{thm:maindiag-B}.
%\end{ootp}

\begin{nota*} Throughout this paper, 
we use the following notation:
\begin{itemize}
\item $\Q^{3}$: the smooth quadric hypersurface in $\P^{4}$. $\mc{O}_{\Q^{3}}(1) \coloneqq \mc{O}_{\P^{4}}(1)|_{\Q^{3}}$.
\item $\Q^{2}_{0}$: the quadric cone in $\P^{3}$. $\mc{O}_{\Q^{2}_{0}}(1) \coloneqq \mc{O}_{\P^{3}}(1)|_{\Q^{2}_{0}}$.
\item $\F_{d}$: the Hirzebruch surface of degree $d$.
\item $f_{d}$: a fiber of $\F_{d}$.
\item $\Sigma_{d}$: the minimal section of $\F_{d}$.
\item $\P_{X}(\mc{E}) \coloneqq \PProj_{\mc{O}_{X}} \oplus_{m \geq 0} \Sym^{m}(\mc{E})$: the projectivization of a locally free sheaf on a variety $X$. We often write it $\P(\mc{E})$ for short.
\item $\mc{O}_{\P(\mc{E})}(1)$: the tautological bundle of a projective bundle $\P(\mc{E})$.
\item $\xi_{\P(\mc{E})}$: the tautological divisor of a projective bundle $\P(\mc{E})$.
\item $\F(a,b,c) \coloneqq \P_{\P^{1}}(\mc{O}_{\P^{1}}(a) \oplus \mc{O}_{\P^{1}}(b) \oplus \mc{O}_{\P^{1}}(c))$.
\item $E_f$: the exceptional divisor of a birational morphism $f$.
\item $\Sing X$: the singular locus of a variety $X$.
\item $Y_{\wt X}$: the strict transformation of a closed subscheme $Y$ of a normal variety $X$ in a birational model $\wt X$ of $X$.
\item $\eu (X)$: the topological Euler number of a topological space $X$.
\item $h^{i,j}(X)$: the dimension of $H^{i}(X, \wedge^{j}\Omega_{X})$ of a smooth projective $3$-fold $X$.
\item $p_{a}(C)$: the arithmetic genus of a smooth projective curve $C$.
\item $\Supp Y$: the support of a closed subscheme $Y$ of an ambient variety.
\item $N_{Y}X$: the normal bundle of a smooth subvariety $Y$ of a smooth variety $X$.
\end{itemize}
\end{nota*}

\begin{conv}
A del Pezzo fibration is an extremal contraction of relative Picard number one from a smooth projective $3$-fold to a smooth projective curve.
The degree of a del Pezzo fibration is the anti-canonical volume of a general fiber, which is a del Pezzo surface.
A $\P^{2}$-bundle (resp.\,a quadric fibration) is a del Pezzo fibration of degree $9$ (resp.\,$8$).

%We say that $f \colon X \ra C$ is a del Pezzo fibration of degree $d$ if $f$ is an extremal ray contraction from a smooth projective $3$-fold $X$ to a smooth projective curve $C$ such that each $f$-fiber is del Pezzo surface of degree $d$. In particular, we assume that the Picard number of $X$ is two. 

%We say that $q \colon Q \ra C$ is a quadric fibration if $q$ is a del Pezzo fibration of degree $8$.

For a surjective morphism $f \colon X \ra Y$ and divisors $D_{1}, D_{2}$ on $X$, the notation $D_{1} \sim_{Y} D_{2}$ means that $D_{1}-D_{2}$ is linearly equivalent to the pullback of some divisor on $Y$.
\end{conv}

\section{Elementary links}\label{sec:elem}

In this section, we have compiled some facts on elementary links, which will be needed in \S \ref{sec:equiv}--\ref{sec:mainB}.
In the following situation, the assumption of Definition \ref{defi:elem} is satisfied, and hence we can construct an elementary link. For the detail, see \S \ref{sec:elembl}--\ref{sec:elemqq}.
\begin{itemize}
\item $\sigma$ is the blow-up at a point and $\tau$ is the blow-up along a curve.
\item $\sigma$ is a $\P^{2}$-bundle and $r$ is a linear subspace of a fiber \cite{Mar}. 
\item $\sigma$ is a quadric fibration and $r$ is a section \cite{D'S}. 
\item $\sigma$ is a quadric fibration and $r$ is a ruling in a $\sigma$-fiber \cite{H-T}.
\end{itemize}

\subsection{Elementary links between blow-ups}\label{sec:elembl}

First we check that the change of the order of the blow-ups at a point and along a curve does not change the output.

\begin{lem}\label{lem:blnorm}
Let $X$ be a smooth 3-fold, $C \subset X$ a smooth irreducible curve and $p \in C$ a point.
Denote by $\vp_{1} \colon X_1 \ra X$ the blow-up at $p$ and by $\vp_{2} \colon X_2 \ra X$ the blow-up along $C$.
Let $C_1$ be the strict transform of $C$ in $X_1$ and $f_{p} \coloneqq \vp_{2}^{-1}(p)$. 
Then the following holds:
\begin{enumerate}
\item[\textup{(1)}] $\Bl_{C_{1}}(\Bl_{p}X) \cong \Bl_{f_{p}}(\Bl_{C}X)$ over $X$.
\item[\textup{(2)}] $N_{C}X \cong N_{C_1} X_{1} \otimes \mc{O}_{C_{1}}(p_{1})$, where $p_{1} \coloneqq E_{\vp_{1}}|_{C_{1}}$.
\end{enumerate}
\end{lem}

\begin{proof}
(1): Let $\psi_{1} \colon \wt X \ra X_{1}$ be the blow-up along $C_{1}$ and $\chi \coloneqq \vp_{1} \circ \psi_{1}$.
Let $E_{p}$ be the strict transform of $E_{\vp_{1}}$ in $\wt X$. 
Then we have $-K_{\wt X} \sim_{X} - 2E_{p} -E_{\psi_{1}}$. 

Consider the divisor $-E_{p} -E_{\psi_{1}}$.
Each irreducible curve $l \subset \wt X$ contracted by $\chi$ is either a fiber of $\psi_{1}|_{E_{\psi_{1}}} \colon E_{\psi_{1}} \ra C_{1}$ or a curve in $E_{p}$. The former satisfies $(l \cdot -E_{p}-E_{\psi_{1}})=1$ and the latter satisfies $(l \cdot -E_{p}-E_{\psi_{1}})=(l \cdot f_{1})_{E_{p}} \geq 0$ regarding $E_{p}$ as $\F_{1}$.
Hence $-E_{p} -E_{\psi_{1}}$ is a $\chi$-nef divisor and $R \coloneqq (-E_{p}-E_{\psi_{1}})^{\perp} \cap \ol\NE(\wt X/X)$ is generated by $f_{1}$ in $E_{p} \cong \F_{1}$. 

Since $(-K_{\wt X} \cdot f_{1})=(-E_{p} \cdot f_{1})>0$, there is the contraction morphism $\psi_{2} \colon \wt X \ra X_{2}'$ of the extremal ray $R$.
Let $\vp_{2}' \colon X_{2}' \to X$ be the induced morphism.
Since the centers of both $\psi_{2}$ and $\vp_{2}'$ is a curve, each of them is the blow-up along a smooth curve by \cite[Theorem 3.3]{Mor}. 
Hence we have $\vp_{2}'= \vp_{2}$ and $\psi_{2}$ is the blow up of $X_{2}$ along $f_{p}$, which proves (1).
%\begin{equation}\label{diag:ptcurve}
%\xymatrix@!C=15pt@R=15pt{
%&\wt X \ar[dl]_-{\psi_{1}} \ar[dr]^-{\psi_{2}}
%&
%\\X_{1} \ar[dr]_-{\vp_{1}}
%&
%&X_{2} \ar[dl]^-{\vp_{2}}
%\\ 
%& X.
%&
%}
%\end{equation}

\noindent(2): It holds that $E_{p}=E_{\psi_{2}}$ and $\psi_{2*}E_{\psi_{1}}=E_{\vp_{2}}$ by (1). Hence we have: 
\begin{eqnarray}\label{eq:ptcurve}
\ \ \ (\psi_{2}|_{E_{\psi_{1}}})^{*}\mc{O}_{\P(N_{C}X^{\vee})}(1)&\!\!\cong\!\!&\mc{O}_{E_{\psi_{1}}}(-\psi_{2}^{*}E_{\vp_{2}})\\
&\!\!\cong\!\!&\mc{O}_{E_{\psi_{1}}}(-E_{\psi_{1}}) \otimes \mc{O}_{E_{\psi_{1}}}(-E_{\psi_{2}})\nonumber\\
&\!\!\cong\!\!&\mc{O}_{\P(N_{C_{1}}X_{1}^{\vee})}(1)\otimes \mc{O}_{E_{\psi_{1}}}(-\psi_{1}^{*}E_{\vp_{1}}).\nonumber\\
&\!\!\cong\!\!&\mc{O}_{\P(N_{C_{1}}X_{1}^{\vee})}(1)\otimes (\psi_{1|_{E_{\psi_{1}}}})^{*}\mc{O}_{C_{1}}(-E_{\vp_{1}}).\nonumber
\end{eqnarray}
Pushing forward (\ref{eq:ptcurve}) by $\chi|_{E_{\psi_{1}}}$, we get $N_{C}X \cong N_{C_{1}}X_{1} \otimes \mc{O}_{C_{1}}(p_{1})$.
\end{proof}

\subsection{Elementary links between $\P^{2}$-bundles}

Elementary links between projective bundles are considered by M.\,Maruyama \cite{Mar} in any dimension.
Here we restrict our attention to $\P^{2}$-bundles.

\begin{lem}[{\cite[Chapter I, \S 2]{Mar}}]\label{lem:p-p}
Let $p \colon P \ra C$ be a $\P^{2}$-bundle and $L \subset P$ a $n$-dimensional linear subspace of a $p$-fiber ($n \leq 1$).
Let $\vp \colon \wt P = \Bl_{L}P \ra P$ be the blow-up along $L$.
Then there exists a divisorial contraction $\psi \colon \wt P \ra P'$ over $C$ such that the induced morphism $p' \colon P' \ra C$ is a $\P^{2}$-bundle and $\psi$ is the blow-up along a $(1-n)$-dimensional linear subspace of a $p'$-fiber $L' \subset P'$. Also it holds that $E_{\psi}$ is the strict transform of the $p$-fiber containing $L$. 
\begin{equation}
\xymatrix@!C=15pt@R=15pt{
&\Bl_{L} P=\wt{P}=\Bl_{L'} P' \ar[dl]_-{\vp} \ar[dr]^-{\psi}
&
\\P \ar[d]_-{p}
&
&P' \ar[d]^-{p'}
\\ C \ar@{=}[rr]
& 
&C.
}
\end{equation}
Moreover, if $\mc{E}$ is an associated vector bundle to $p \colon P \to C$, then we can take a vector bundle $\mc{E}'$ associated to $p' \colon P' \to C$ such that $\deg\mc{E}'=\deg\mc{E}-(n+1)$. 
\end{lem}

\begin{cor}\label{cor:p-p}
We follow the notation of Lemma \ref{lem:p-p}. 
Suppose that $C \cong \P^{1}$.
Let $F$ be a $p$-fiber and $D$ a sub $\P^{1}$-bundle of $P$. Take $a \in \Z$ such that $D \sim \xi_{P} + aF$. 
Then the following hold:
\begin{equation}
  \begin{cases}
     D_{P'} \sim\xi_{P'} + (a+1)F \text{ and } L' \subset D_{P'}&  \text{if } L \not\subset D,\\
     D_{P'} \sim\xi_{P'} + aF       \text{ and } L' \not\subset D_{P'}&  \text{if } L \subset D.
  \end{cases}
\end{equation}
\end{cor}

\begin{proof}
By the canonical bundle formula, we have:
\begin{eqnarray}\label{eq:ppcor1}
-K_{P} &\sim& 3\xi_{P}-(\deg\mc{E}-2)F,\\
-K_{P'} &\sim& 3\xi_{P'}-(\deg\mc{E}-(3+n))F.
\end{eqnarray}
Also it holds that 
\begin{equation}\label{eq:ppcor2}
-K_{\wt P} \sim \vp^{*}(-K_{P})-(2-n)E_{\vp} \sim \psi^{*}(-K_{P'})-(n+1)E_{\psi}.
\end{equation}
Combining (\ref{eq:ppcor1})--(\ref{eq:ppcor2}) and $E_{\vp} \sim F-E_{\psi}$, we have $3\vp^{*}\xi_{P} \sim 3\psi^{*}\xi_{P'}+3(F-E_{\psi}).$ Since $\Pic \wt P$ is torsion-free, it holds that:
\begin{equation}\label{eq:ppcor3}
\vp^{*}\xi_{P} \sim \psi^{*}\xi_{P'}+F-E_{\psi}.
\end{equation}

On the other hand, we have:
\begin{equation}
D_{\wt P} \sim
  \begin{cases}
     \vp^{*}\xi_{P} + aF &  \text{if } L \not\subset D,\\
     \vp^{*}\xi_{P} + aF -E_{\vp} &  \text{if } L \subset D\label{eq:ppcor4}.
  \end{cases}
\end{equation} 
Combining (\ref{eq:ppcor3}), (\ref{eq:ppcor4}) and $E_{\vp} \sim F-E_{\psi}$, we have:
\begin{equation}
D_{\wt P} \sim
  \begin{cases}
     \psi^{*}\xi_{P'} + (a+1)F-E_{\psi} &  \text{if } L \not\subset D,\label{eq:ppcor5}\\
     \psi^{*}\xi_{P'} + aF &  \text{if } L \subset D.
  \end{cases}
\end{equation} 
By pushing forward (\ref{eq:ppcor5}) by $\psi$, we have the assertion.
\end{proof}

\subsection{Elementary links from quadric fibrations to $\P^{2}$-bundles}
H. D'Souza \cite{D'S} showed the existence of elementary links from quadric fibrations to $\P^{2}$-bundles.
The precise statement is as follows.

\begin{lem}[{\cite[(2.7.3)]{D'S}, \cite[Proposition 3.1]{Fuk}}]\label{lem:q-p}
Let $q \colon Q \ra C$ be a quadric fibration and $s \subset Q$ a $q$-section. 
Let $\vp \colon \wt{Q}=\Bl_{s} Q \ra Q$ be the blow-up of $Q$ along $s$. 
Then there exists a divisorial contraction $\psi \colon \wt{Q} \ra P$ over $C$ such that the induced morphism $p \colon P \ra C$ is a $\P^{2}$-bundle and $\psi$ is the blow-up along a smooth $p$-bisection $B \subset P$.
\begin{equation}\label{diag:q-p}
\xymatrix@!C=15pt@R=15pt{
&\Bl_{s} Q=\wt{Q}=\Bl_{B} P \ar[dl]_-{\vp} \ar[dr]^-{\psi}
&
\\Q \ar[d]_-{q}
&
&P \ar[d]^-{p}
\\ C \ar@{=}[rr]
& 
&C.
}
\end{equation}
Moreover, let $H_{Q}$ be a $q$-ample divisor with $2H_{Q} \sim_{C} -K_{Q}$ and $H_{P}$ a $p$-ample divisor such that $3H_{P} \sim_{C} -K_{P}$. Then:
\begin{enumerate}
\item[\textup{(1)}] It holds that $E_\psi \sim_{C} \vp^{*}H_{Q}-2E_{\vp}$ and $E_\vp \sim_{C} \psi^{*}H_{P}-E_\psi$.
\item[\textup{(2)}] The branched locus of $p|_{B}$ coincides with the closed set 
\begin{equation}
\Sigma \coloneqq \{t \in C \mid q^{-1}(t) \mbox{ is singular }\}.
\end{equation}
\item[\textup{(3)}] It holds that $(-K_{Q})^{3}=40-(8p_{a}(B)+32p_{a}(C))$.
\end{enumerate}
\end{lem}

\begin{lem}\label{lem:fiber}
We follow the notation of Lemma \ref{lem:q-p}. Let $E \coloneqq \psi_{*}(E_\vp)$. 
Suppose that $H_{Q}$ is a prime divisor containing $s$ and assume that $H_{Q}$ is normal. 
Then $(H_{Q})_{\wt Q} \sim_{C} \psi^{*} H_{P}$.

Moreover, when $H_{P} = (H_{Q})_{P}$, the following holds for $t \in C$. 
\begin{enumerate}
\item[\textup{(1)}] If $t \not\in \Sigma$, then $(q|_{H_{Q}})^{-1}(t)$ is
\begin{eqnarray*}
\mathrm{smooth} &\iff& p^{-1}(t) \cap B \cap H_{P}= \emptyset.\\
\mathrm{reducible} &\iff& p^{-1}(t) \cap B \cap H_{P} \neq  \emptyset.\nonumber
\end{eqnarray*}
\item[\textup{(2)}] If $t \in \Sigma$, then $(q|_{H_{Q}})^{-1}(t)$ is
\begin{eqnarray*}
\mathrm{smooth} &\iff& p^{-1}(t) \cap B \cap H_{P}= \emptyset.\\
\mathrm{reducible} &\iff& p^{-1}(t) \cap B \cap H_{P} \neq \emptyset \mathrm{\ and\ } E|_{p^{-1}(t)} \neq H_{P}|_{p^{-1}(t)}.\nonumber\\
\mathrm{non\mbox{-}reduced} &\iff& E|_{p^{-1}(t)}=H_{P}|_{p^{-1}(t)}.\nonumber
\end{eqnarray*}
\end{enumerate}
\end{lem}
\begin{proof}
the first assertion follows from Lemma \ref{lem:q-p} (1).
we take the following diagram as the base change of (\ref{diag:q-p}) at $t \in C$:
\begin{equation}\label{diag:qpfiber}
\xymatrix@!C=15pt@R=15pt{
&\Bl_{s_{t}} Q_{t}=\wt{Q_{t}}=\Bl_{B_{t}} P_{t} \ar[dl]_-{\vp_{t}} \ar[dr]^-{\psi_{t}}
&
\\Q_{t}
&
&P_{t}. 
}
\end{equation}
Write $G_{t} \coloneqq (\vp_{*} E_{\psi})|_{Q_{t}}$, $H_{t} \coloneqq H_{Q}|_{Q_{t}}$, $s_{t} \coloneqq s|_{Q_{t}}$ and $B_{t} \coloneqq B|_{P_{t}}$.

\noindent(1): In this case we have $Q_{t} \cong \F_{0}$.
By lemma \ref{lem:q-p} (1),  it holds that $H_{t} \sim G_{t} \sim \Sigma_{0}+f_{0}$ and $G_{t}$ is the union of two rulings containing $s_{t}$. 
Since $H_{t}$ is smooth if and only if $H_{t}$ is irreducible, we only have to show that $H_{t}$ is smooth $\Rightarrow (H_{t})_{P_{t}} \cap B_{t}= \emptyset \Rightarrow H_{t}$ is irreducible.

Let $G_{t} =G_{1} + G_{2}$ be the irreducible decomposition. 
Note that $(H_{t} \cdot G_{i})_{Q_{t}}=1$ for $i=1,2$.
Suppose that $H_{t}$ is smooth. 
Then $H_{t} \cap G_{i} = s_{t}$ scheme-theoretically for $i=1,2$. 
Hence we have $(H_{t})_{\wt Q_{t}} \cap E_{\psi_{t}} =\emptyset$ and $(H_{t})_{P_{t}} \cap B_{t}= \emptyset$.
On the other hand, if $(H_{t})_{P_{t}} \cap B_{t}= \emptyset$, then $(H_{t})_{\wt Q_{t}} \cong (H_{t})_{P_{t}}$ is irreducible and so is $H_{t}$, and (1) is proved.

\noindent(2):
In this case, $\psi_{t}$ is a weighted blow-up and 
hence $Q_{t} \cong \Q^{2}_{0}$.
Since $(s \cdot Q_{t})_{Q}=1$, the point $s_{t}$ is not the vertex of $Q_{t} \cong \Q^{2}_{0}$. 
By lemma \ref{lem:q-p} (1), it holds that $H_{t} \sim G_{t} \sim \mc{O}_{\Q^{2}_{0}}(1)$, and we have $G_{t} = 2l'$, where $l'$ is the unique ruling of $\Q^{2}_{0}$ containing $s_{t}$.

Suppose that $H_{t}$ is smooth. 
Since $H_{t} \cap l' =s_{t}$ scheme-theoretically, we have $(H_{t})_{\wt Q_{t}} \cap E_{\psi_{t}}=\emptyset$ and hence $(H_{t})_{P_{t}} \cap B_{t}= \emptyset$.

Suppose that $H_{t}$ is reducible.
Then $H_{t}$ is the union of two distinct ruling of $Q_{t} \cong \Q^{2}_{0}$.
Since $s_{t} \in H_{t}$, there exists a ruling $l \neq l'$ of $Q_{t}$ such that $H_{t}=l+l'$.
Since $H_{t}$ is smooth at $s_{t}$, it holds that $(H_{t})_{\wt Q_{t}}=l_{\wt Q_{t}}+E_{\psi_{t}}$.
Hence $(H_{t})_{P_{t}}=l_{P_{t}}$ contains $\Supp B_{t}$ but $(H_{P})|_{P_{t}} \neq E|_{P_{t}}$

Suppose that $H_{t}$ is non-reduced.
Then $\Supp H_{t}$ is a ruling of $Q_{t} \cong \Q^{2}_{0}$.
Since $s_{t} \in H_{t}$, we have $H_{t}=G_{t}$ and hence $H_{P}|_{P_{t}} = E|_{P_{t}}$.
%Hence $H_{P}|_{P_{t}} = E|_{P_{t}}$ because each of them is a line in $P_{t}$ by Lemma \ref{lem:q-p} (2)

Combining these results, we complete the proof.
\end{proof}

\subsection{Elementary links between quadric fibrations}\label{sec:elemqq}

B.\,Hassett and Y.\, Tschinkel \cite{H-T} considered elementary links between quadric fibrations with center a ruling in a smooth fiber. 
We can prove that a similar elementary link appears in the case of a singular fiber as follows.

\begin{lem}\label{lem:q-q}
Let $q \colon Q \ra C$ be a quadric fibration and $l$ a ruling of a $q$-fiber.
Let $\vp \colon \wt{Q}=\Bl_{l} Q \ra Q$ be the blow-up of $Q$ along $l$. 
Then there exists a divisorial contraction $\psi \colon \wt{Q} \ra Q'$ over $C$ such that the induced morphism $q' \colon Q \ra C$ is a quadric fibration and $\psi$ is the blow-up along a ruling $l'$ of a $q'$-fiber.
\begin{equation}
\xymatrix@!C=15pt@R=15pt{
&\Bl_{l} Q=\wt{Q}=\Bl_{l'} Q' \ar[dl]_-{\vp} \ar[dr]^-{\psi}
&
\\Q \ar[d]_-{q}
&
&Q' \ar[d]^-{q'}
\\ C \ar@{=}[rr]
& 
&C.
}
\end{equation}
\end{lem}

\begin{proof}
Let $F \subset Q$ be the $q$-fiber containing $l$. When $F$ is smooth, then the assertion is already shown by \cite[\S 5]{H-T}.
Hence we may assume that $F \cong \Q^{2}_{0}$. 

First we calculate $N_{l}Q$. 
Let $v \in F$ be the vertex of $\Q^{2}_{0}$ and $h \colon Q_{1} \ra Q$ the blow-up at $v$.
Let $F_{1}$ (resp.\ $l_{1}$) be the strict transform of $F$ (resp.\ $l$) in $Q_{1}$. 
Then $l_{1}$ is a fiber of $F_{1} \cong \F_{2}$ and $F_{1} \sim h^{*}F -2E_ h$. 
Since $N_{l_{1}}F_{1} \cong \mc{O}_{l_{1}}$ and $(N_{F_{1}}Q_{1})|_{l_{1}} \cong \mc{O}_{l_{1}}((F_{1} \cdot l_{1})) \cong \mc{O}_{l_{1}}(-2)$, we have $N_{l_{1}}Q_{1} \cong \mc{O}_{l_{1}} \oplus \mc{O}_{l_{1}}(-2)$ by the normal bundle sequence. 
Hence $N_{l}Q=\mc{O}_{l}(1) \oplus \mc{O}_{l}(-1)$ by Lemma \ref{lem:blnorm} (2)

Therefore we have $E_{\vp} \cong \F_{2}$ and $E_{\vp}|_{E_{\vp}} \sim -(\Sigma_{2} +f_{2})$. 
Since $F_{1} \cong \F_{2}$, it follows that $F_{\wt Q} \cong \F_{2}$ from Lemma \ref{lem:blnorm} (1).
Since $(E_{\vp}+F_{\wt Q})|_{E_{\vp}}=\vp^{*}F|_{E_{\vp}} \sim 0$, we have $F_{\wt Q}|_{E_{\vp}} \sim \Sigma_{2} +f_{2}$.
Hence $F_{\wt Q}|_{E_{\vp}}$ is the sum of $C_{1} \sim \Sigma_{2}$ and $C_{2} \sim f_{2}$.
On the other hand, in $F_{\wt Q}$, we have $C_{1} \sim f_{2}$ and $C_{2} \sim \Sigma_{2}$ because $F_{\wt Q}$ is the minimal resolution of $F$.
By symmetry of $E_{\vp}$ and $F_{\wt Q}$, there is the blow-down of $F_{\wt Q}$ as desired.
\end{proof}

The following is the key to proving Theorem \ref{thm:maindiag-A} (1).

\begin{lem}\label{lem:nnorm}
Let $q \colon Q \ra C$ be a quadric fibration and $D_{h} \subset Q$ a prime divisor such that $2D_{h} \sim_{C} -K_{Q}$.
Suppose that $D_{h}$ is non-normal. Let $R$ be the $1$-dimensional component of $\Sing D_{h}$. Then:
\begin{enumerate}
\item[\textup{(1)}] $R$ is a $q$-section.
\item[\textup{(2)}] If we take the elementary link $Q \xla{\vp} \wt P \xra{\psi} P$ with center along $R$, then we have $D_{h}=(E_{\psi})_{Q}$. In particular, we have $R=\Sing D_{h}$.
%Moreover $D_{h}$ is the unique prime divisor which is linearly equivalent to $-\frac{1}{2}K_{Q}$ over $C$ and is singular along $R$.
\end{enumerate}
\end{lem}

\begin{proof}
\noindent(1): Let $r$ be an irreducible component of $R$. 
To seek a contradiction, assume that $q(r)$ is a point. Take a $q$-fiber $F$ containing $r$. 
Since $D_{h}$ is singular along $r$, the restriction $D_{h}|_{F}$ is non-reduced along $r$.
If $F \cong \F_{0}$, then $D_{h}|_{F}$ is reduced since $D_{h}|_{F} \sim \Sigma_{0} +f_{0}$, a contradiction.
Therefore $F \cong \Q^{2}_{0}$ and there is a ruling $r \subset F$ such that $D_{h}|_{F}=2r$.

Let $\chi \colon \wt Q \ra Q$ be the blow-up along $r$. In the proof of Lemma \ref{lem:q-q}, we have shown that $F_{\wt Q} \cong \F_{2}$ and $F_{\wt Q}|_{F_{\wt Q}} \sim -E_{\chi}|_{F_{\wt Q}} \sim -(\Sigma_{2}+f_{2})$. 
Hence we have:
\begin{eqnarray}
\ \ \ \ \ \ \ ((D_{h})_{\wt Q}^{2} \cdot F_{\wt Q})&=&((\chi^{*}D_{h}-2 E_{\chi})^{2} \cdot (\chi^{*}F-E_{\chi}))\label{eq:2-4-2}\\
                                       &=&(D_{h}^{2} \cdot F)_{Q} -4(D_{h} \cdot r)_{Q} -4(F \cdot r)_{Q} -4E_{\chi}^{3}=-2,\nonumber\\
\ \ \ \ \ \ \ ((D_{h})_{\wt Q} \cdot F_{\wt Q}^{2})&=&((\chi^{*}D_{h}-2 E_{\chi}) \cdot (\chi^{*}F-E_{\chi})^{2})\label{eq:2-4-3}\\
                                       &=&(D_{h} \cdot F^{2})_{Q}-(D_{h} \cdot r)_{Q}-4(F \cdot r)_{Q}-2E_{\chi}^{3}=-1.\nonumber
\end{eqnarray}
Take $a, b \in \Z$ such that $(D_{h})_{\wt Q}|_{F_{\wt Q}} \sim a\Sigma_{2}+bf_{2}$. By (\ref{eq:2-4-2}) and (\ref{eq:2-4-3}), we have $-2a^{2}+2ab=-2$ and $a-b=-1$. Hence $(a, b)=(-1, 0)$, which is absurd.

Therefore $r$ dominates $C$. 
Let $F$ be a smooth $q$-fiber. 
Then we have $\emptyset \neq \Supp(R \cap F) \subset \Sing (D_{h}|_{F})$.
Since $D_{h}|_{F} \sim \Sigma_{0}+f_{0}$, it follows that $\Supp(R \cap F)=\Sing (D_{h}|_{F})$ is a point and hence $R$ is a $q$-section.

\noindent(2): By Lemma \ref{lem:q-p} (1), $(E_{\psi})_{Q}$ is singular along $R$.
For each smooth $q$-fiber $F$, there is the unique member of $|\Sigma_{0}+f_{0}|$ singular at $\Supp(R \cap F)$. Hence $D_{h}|_{F}=(E_{\psi})_{Q}|_{F}$, and the first assertion follows.
Since $\vp$ is the blow-up along $R$ and $E_{\psi}$ is smooth, the last assertion follows.
\end{proof}

\section{Topological invariants of the ambient space}\label{sec:top}

In this section, we determine the Hodge diamonds of del Pezzo fibrations containing affine homology $3$-cells, and that of the base curves.

\begin{lem}\label{lem:sub-1}
Let $f \colon X \to C$ be a del Pezzo fibration and $D$ a reduced effective divisor on $X$ such that $X \setminus D$ is an affine homology $3$-cell. Then the Hodge diamond of $X$ is as follows:
\begin{equation*}
{\tiny\xymatrix@=-1ex{
&&& 1 \\
&& 0 && 0 \\
& 0 && 2 && 0 \\
0\ \ \  && h^{1,2}(X) && h^{1,2}(X) && \ \ \ 0\\
& 0 && 2 && 0 \\
&& 0 && 0 \\
&&& 1 
}}
\end{equation*}
Moreover, It holds that $C \cong \P^{1}$.
\end{lem}

\begin{proof}
By the Hodge symmetry, we only have to compute $h^{i,0}(X)$ for $1 \leq i \leq 3$ and $h^{1,1}(X)$.
Since $-K_{X}$ is $f$-ample, we have the following by the relative Kawamata-Viehweg vanishing theorem:
\begin{equation}\label{eq:KV}
h^{i}(X, \mc{O}_{X})=h^{i}(C, \mc{O}_{C}) \mathrm{\ for\ } i \geq 0.
\end{equation} 
In particular, we have $h^{2,0}(X)=h^{3,0}(X)=0$.
Since the Picard number of $X$ is two by assumption, we have $h^{1,1}(X)=h^{1,1}(X)+2h^{2,0}(X)=2$. 

On the other hand, by \cite[Proposition 2.1]{van}, we have $H^{5}(X, \Z) \cong H^{5}(D, \Z)$ $=0$.
Hence $H^{1}(X, \Z)=0$ by the Poincare duality and $h^{1,0}(X)=0$ by the Hodge decomposition, which proves the first assertion.
The second assertion follows from (\ref{eq:KV}).
\end{proof}

\section{Proof of Theorem \ref{thm:main0}.}\label{sec:equiv}

This section is devoted to the proof of Theorem \ref{thm:main0}.
First we determine the linear equivalence class of the irreducible components of the boundary divisor.
Then we give the precise statement of Theorem \ref{thm:main0} as in Theorem \ref{thm:main1} and prove it by using Lemma \ref{lem:q-p}, i.e.\ elementary links from quadric fibrations to $\P^{2}$-bundles.

\begin{lem}\label{lem:sub-2}
Let $q \colon Q \to C$ be a quadric fibration, $D_{h}$ a reduced effective divisor on $Q$, and $D_{f}$ a $q$-fiber.
If $U \coloneqq Q \setminus (D_{h} \cup D_{f})$ is an affine homology $3$-cell, then $D_{h}$ is a prime divisor such that $2D_{h} \sim_{C} -K_{Q}$.
\end{lem}

\begin{proof}
By Lemma \ref{lem:sub-1}, it follows that $\Pic_{0}(Q)=0$. 
By \cite[Corollary 1.20]{Fuj}, we have $\Pic U = 0$ and the group of invertible functions on $U$ coincides with non-zero constants $\C^{*}$.
Hence $D_{h}$ is a prime divisor such that $\Pic Q = \Z D_{f} \oplus \Z D_{h}$. By \cite[Theorem 3.5]{Mor} and the Grothendieck-Lefschetz theorem, there exists a divisor $H_{Q}$ on $Q$ such that $\Pic Q = \Z D_{f} \oplus \Z H_{Q}$ and $2H_{Q} \sim_{C} -K_{Q}$. Hence $D_{h} \sim_{C} H_{Q}$, which proves the lemma.
\end{proof}

The following is the precise statement of Theorem \ref{thm:main0}.

\begin{thm}\label{thm:main1}
Let $q \colon Q \ra C$ be a quadric fibration, $D_{h}$ a reduced effective divisor on $Q$, and $D_{f}$ a $q$-fiber.
\begin{enumerate}
\item[\textup{(A)}] Suppose that $D_{h}$ is non-normal.
Then the following are equivalent.
\begin{enumerate}
\item[\textup{(1)}] The complement $Q \setminus (D_{h} \cup D_{f})$ is an affine homology $3$-cell. 
\item[\textup{(2)}] $C \cong \P^{1}$ and $D_{h}$ is a prime divisor such that $2D_{h} \sim_{C} -K_{Q}$.
\item[\textup{(3)}] It holds that $Q \setminus (D_{h} \cup D_{f}) \cong \A^{3}$.
\end{enumerate}

\item[\textup{(B)}] Suppose that $D_{h}$ is normal.
Then the following are equivalent.
\begin{enumerate}
\item[\textup{(1)}] The complement $Q \setminus (D_{h} \cup D_{f})$ is an affine homology $3$-cell. 
\item[\textup{(2)}] $C \cong \P^{1}$ and $D_{h}$ is a prime divisor such that $2D_{h} \sim_{C} -K_{Q}$. Also we have $D_{f} \cong \Q^{2}_{0}$ and $h^{1,2}(Q)=0$. Moreover, each $(q|_{D_{h}})$-fiber is smooth except possibly $D_{f}|_{D_{h}}$.
\item[\textup{(3)}] It holds that $Q \setminus (D_{h} \cup D_{f}) \cong \A^{3}$.
\end{enumerate}
\end{enumerate}
\end{thm}

\begin{proof}
\noindent(A):
Since $(3) \Rightarrow (1)$ is trivial and $(1) \Rightarrow (2)$ follows from Lemma \ref{lem:sub-1} and Lemma \ref{lem:sub-2}, we only have to show $(2)\Rightarrow (3)$.

Suppose that (2) holds. 
Let $s \coloneqq \Sing D_{h}$, which is a $q$-section by Lemma \ref{lem:nnorm}.
Construct $\vp, \psi, p, \wt Q$ and $P$ as in Lemma \ref{lem:q-p}.
Then we have $(E_{\psi})_{Q}=D_{h}$ by Lemma \ref{lem:nnorm} (2). 
Therefore we have:
\begin{equation}\label{eq:main1'}
Q \setminus (D_{h} \cup D_{f}) \cong \wt Q \setminus ((D_{h})_{\wt Q} \cup (D_{f})_{\wt Q} \cup E_\vp) \cong P \setminus ((D_{f})_{P} \cup (E_\vp)_{P}).
\end{equation}
Since $(E_\vp)_{P}$ is a sub $\P^{1}$-bundle by Lemma \ref{lem:q-p} (1) and $(D_{f})_{P}$ is a $p$-fiber, 
we have $Q \setminus (D_{h} \cup D_{f}) \cong \A^{3}$ by \cite[Lemma 5.15]{Kis}.

\noindent(B):
Since $(3) \Rightarrow (1)$ is trivial, we only have to show $(1) \Rightarrow (2) \Rightarrow (3)$.
Let $U \coloneqq Q \setminus (D_{h} \cup D_{f})$. 
We note that (1) implies $C \cong \P^{1}$ by Lemma \ref{lem:sub-1}. 
Hence we may assume that $C \cong \P^{1}$ throughout the proof.
Let $\infty \coloneqq q(D_{f})$ and regard $C \setminus \{\infty\}$ as $\A^{1}$.

\noindent$(1) \Rightarrow (2)$: 
Suppose that (1) holds. Then the second assertion follows from Lemma \ref{lem:sub-1}.
Since $U$ is an affine homology $3$-cell, we have:
\begin{eqnarray}\label{eq:main1-1}
\eu(Q)&=&\eu(U)+\eu(D_{h} \setminus (D_{f}|_{D_{h}}))+\eu(D_{f})\\
          &=&1+\eu(D_{h} \setminus (D_{f}|_{D_{h}}))+\eu(D_{f}). \nonumber
\end{eqnarray}

Let $\sigma \coloneqq \{ t \in \A^{1} \mid (q|_{D_{h}})^{*}(t) \mbox{ is reducible}\}$.
For $t \in \A^{1}$, the divisor $(q|_{D_{h}})^{*}(t)$ is a member of either $|\Sigma_{0}+f_{0}|$ in $\F_{0}$ or $|\mc{O}_{\Q^{2}_{0}}(1)|$ in $\Q^{2}_{0}$.
In particular, we have:
\begin{eqnarray*}
t \not \in \sigma &\iff& \Supp \ (q|_{D_{h}})^{*}(t) \cong \P^{1}\\
                          &\iff& \eu((q|_{D_{h}})^{*}(t))=2.\\
 t \in \sigma       &\iff& q^{*}(t) \cong \F_{0} \mbox{ and } (q|_{D_{h}})^{*}(t) \mbox{ is reducible}  \\
                          &\iff& \eu((q|_{D_{h}})^{*}(t))=3.
\end{eqnarray*}
Hence we have:
\begin{equation}\label{eq:main1-2}
\eu(D_{h} \setminus (D_{f}|_{D_{h}}))=2\eu(\A^{1} \setminus \sigma)+3\eu(\sigma)=2+\sharp \sigma.
\end{equation}
Also by Lemma \ref{lem:sub-1}, we have:
\begin{equation}\label{eq:main1-3}
\eu(Q) = 6-2h^{1,2}(Q).
\end{equation}
Combining (\ref{eq:main1-1})--(\ref{eq:main1-3}), we have:
\begin{equation}\label{eq:main1-4}
6-2h^{1,2}(Q) \geq 3+\sharp \sigma +\eu(D_{f}).
\end{equation}
we note that $\eu(D_{f})=3$ when $D_{f} \cong \Q^{2}_{0}$ and $\eu(D_{f})=4$ when $D_{f} \cong \F_{0}$. Hence (\ref{eq:main1-4}) implies that $h^{1,2}(Q)=0$, $\sigma=\emptyset$ and $D_{f} \cong \Q^{2}_{0}$. 
In particular, we get the third and fourth assertion of (2).

It remains to prove the last assertion.
Take a $q$-section $s \subset D_{h}$ and construct $\vp, \psi, p, \wt Q, P$ and $B$ and as in Lemma \ref{lem:q-p}.
By (\ref{diag:q-p}), we have: 
\begin{equation}\label{eq:main1-5}
\eu(Q)=6-2p_{a}(B).
\end{equation}
Combining (\ref{eq:main1-3}) and (\ref{eq:main1-5}), we have $p_{a}(B)=h^{1,2}(Q)=0$.
In particular, the branch locus of $p|_{B}$ consists of two points.
By Lemma \ref{lem:q-p} (2), there is exactly two singular $q$-fibers. 
Since $q^{*}(\infty)=D_{f} \cong \Q^{2}_{0}$, we may assume that $q^{*}(0)$ is the other singular fiber.
By Lemma \ref{lem:fiber} and the fact that $\sigma = \emptyset$, each $(q|_{D_{h}})$-fiber is smooth except possibly $D_{f}|_{D_{h}}$ and $(q|_{D_{h}})^{*}(0)$.
Hence we only have to show that $(q|_{D_{h}})^{*}(0)$ is smooth.

Conversely, suppose that $(q|_{D_{h}})^{*}(0)$ is not smooth. 
Let $E \coloneqq \psi_{*}(E_\vp)$ and $U' \coloneqq P \setminus ((D_{h})_{P} \cup(D_{f})_{P}) \cong \A^{3}$.
Since $U \cong \wt Q \setminus ((D_{h})_{\wt Q} \cup(D_{f})_{\wt Q} \cup E_\vp)$, we can regard $U$ as the affine modification of $U'$ with the locus $(B \cap U' \subset E \cap U')$ (see \cite{K-Z} for the definition).
By \cite[Theorem 3.1]{K-Z}, the morphism between homologies $\tau \colon H_{1}(B \cap U', \Z) \ra H_{1}(E \cap U', \Z)$ induced by the inclusion $B \cap U' \hookrightarrow E \cap U'$ is an isomorphism of $\Z$-modules.

On the other hand, $(q|_{D_{h}})^{*}(0)$ is non-reduced because $\sigma = \emptyset$.
Lemma \ref{lem:fiber} now shows that $E \cap U' \cong \A^{1} \times \C^{*}$ and $B \cap U' \cong \C^{*}$, which is an unramified 2-section of the second projection of $E \cap U'$. Hence  $H_{1}(B \cap U', \Z) \cong H_{1}(E \cap U', \Z) \cong \Z$, but $\tau= 2 \times \mathrm{id}_{\Z}$, a contradiction.

\noindent$(2) \Rightarrow (3)$: 
Suppose that (2) holds.
Let $E \coloneqq \psi_{*}(E_\vp)$ and $U' \coloneqq P \setminus ((D_{h})_{P} \cup(D_{f})_{P}) \cong \A^{3}$. 
Then we can regard $U$ as the affine modification of $U'$ with the locus $(B \cap U' \subset E \cap U')$.
By Lemma \ref{lem:fiber}, we have $E \cap U' \cong \A^{2}$ and $B \cap U' \cong \A^{1}$.
By the Abhyanker-Mor theorem over Noetherian rings containing $\Q$ \cite[Theorem B]{B-D}, 
there is a coordinate $\{x, y, z\}$ of $U'=\A^{3}$ such that $E \cap U'=\{x=0\}$ and $B \cap U'=\{x=y=0\}$.
Hence $U$ is isomorphic to the affine modification of $\A^{3}_{[x, y, z]}$ with the locus $(\{x=y=0\} \subset \{x=0\})$, which is isomorphic to $\A^{3}$ as desired.
\end{proof}

\section{Examples}\label{sec:ex}

This section provides several examples of compactifications of affine homology $3$-cells compatible with quadric fibrations.
For the construction, we often use Theorem \ref{thm:main1}. 
Throughout this section, $(Q, D_{h}, D_{f})$ stands for a compactification of $\A^{3}$ compatible with a quadric fibration $q \colon Q \to \P^{1}$.
We note that $K_{Q}+D_{h}+D_{f}$ is not nef since $(K_{Q}+D_{h}+D_{f} \cdot l)=-1$ for each ruling $l$ of a $q$-fiber. 

First suppose that $Q$ is a Fano $3$-fold and $D_{h}+D_{f}$ is ample.
Let us mention that then in \cite[Lemma 5.9]{Kis}, $D_{h}$ is erroneously claimed to be normal.
In Example \ref{ex:2}, we construct examples with non-normal $D_{h}$.

%In \cite{Kis}, Kishimoto gave a complete classification of compactifications $(X, D)$ of contractible affine 3-folds into smooth Fano 3-folds $X$ with $B_2(X)=2$ such that $K_X+D$ is not nef.
%Unfortunately three cases when $X$ is a quadric fibration and $D$ is ample were erroneously omitted.
%Here we would like to construct examples of these cases.

\begin{ex}\label{ex:2}
Let $q \colon Q \to \P^1$ be a Fano quadric fibration, i.e.\,either No.\,18, No.\,25 or No.\,29 in \cite[Table 2]{M-M}. 
Let $D_{f}$ be a $q$-fiber.
By \cite[Theorem 4.2]{Man}, we can take a $q$-section $s$. 
By Lemma \ref{lem:q-p} (1), there is a prime divisor $D_{h}$ on $Q$ such that $2D_{h} \sim_{\P^{1}} -K_{Q}$ and $\Sing D_{h}=s$.
Theorem \ref{thm:main1} (A) now shows that $(Q, D_{h}, D_{f})$ is a compactification of $\A^{3}$ compatible with \nolinebreak$q$.

Assume that $D_{h}+D_{f}$ is not ample. 
Then by \cite[Lemma 2.2]{Kis} there is a birational extremal contraction $\vp$ of $Q$ such that $E_{\vp}=D_{h}$ or $D_{f}$. 
Since $D_{f}$ is a $q$-fiber, we have $E_{\vp}=D_{h}$, which is impossible since $D_{h}$ is non-normal and $E_{\vp}$ is normal by \cite[Theorem 3.3]{Mor}. 
Hence $D_{h}+D_{f}$ is ample.
\end{ex}

Secondly, suppose that $D_{h}$ is normal and $Q$ is No.\,29 in \cite[Table 2]{M-M}, i.e.\,the blow-up of $\Q^{3}$ along a smooth conic.
Let us mention that then in \cite[Lemma 5.13]{Kis}, $D_{h}+D_{f}$ erroneously claimed to be not ample.
%This is due to using the same arguments in \cite[\S 4.4, Lemma 2]{Mul} where we assume that $D_{f}|_{D_{f}}$ is smooth. 
In Example \ref{ex:3}, we construct an example with $D_{h}+D_{f}$ ample.
%such that $Q$ is No.\,29 in [\textit{ibid.}], $D_{h}$ is normal and $D_{h}+D_{f}$ is ample.

\begin{ex}\label{ex:3}
Take $H, S$ and $C$ in $\Q^{3} = \{X_{0}X_{1}+X_{2}^{2}+X_{3}X_{4}=0\} \subset \P^{4}_{[X_{0}: \cdots:X_{4}]}$ as follows:
\begin{eqnarray}
H &\coloneqq&\{X_{0}X_{1}+X_{2}^{2}+X_{3}X_{4}=X_{0}=0\},\\
S &\coloneqq&\{X_{0}X_{1}+X_{2}^{2}+X_{3}X_{4}=X_{1}X_{3}+X_{0}^{2}=0\},\\
C       &\coloneqq&\{X_{0}X_{1}+X_{2}^{2}+X_{3}X_{4}=X_{0}=X_{1}=0\}.
\end{eqnarray}

Let $P \coloneqq \{X_{0}Y_{1}=X_{1}Y_{0}\} \subset \P^{4}_{[X_{0}: \cdots:X_{4}]} \times \P^{1}_{[Y_{0}:Y_{1}]}$, and $\Phi \colon P \ra \P^{4}$ be the blow-up along $\{X_{0}=X_{1}=0\}$.
Set $Q$, $D_{f}$ and $D_{h}$ as the strict transformations of $\Q^{3}$, $H$ and $S$ in $P$ respectively. 
Then $\Phi|_{Q} \colon Q \ra \Q^{3}$ is the blow-up along $C$, 
and the second projection of $\P^{4} \times \P^{1}$ induces a quadric fibration $q \colon Q \ra \P^{1}_{[Y_{0}:Y_{1}]}$.
The defining equations of $Q$, $D_{f}$ and $D_{h}$ in $P$ are as follows:
\begin{eqnarray}
Q         &=&\{X_{0}X_{1}+X_{2}^{2}+X_{3}X_{4}=0\},\\
D_{f}    &=&\{X_{0}X_{1}+X_{2}^{2}+X_{3}X_{4}=Y_{0}=0\},\\
D_{h} &=&\{X_{0}X_{1}+X_{2}^{2}+X_{3}X_{4}=Y_{1}X_{3}+Y_{0}X_{0}=0\}.
\end{eqnarray}
Then $D_{f}$ is a singular $q$-fiber and $D_{h}$ has only one DuVal singularity of type $D_{4}$. 
Also $D_{h}$ is a prime divisor with $2D_{h} \sim_{\P^{1}}-K_{Q}$. 
Since $C \cong \P^{1}$, we have $h^{1,2}(Q)=0$.
Also we have:
\begin{eqnarray*}
D_{h} \setminus (D_{f}|_{D_{h}}) 
&\cong& 
\left\{
\begin{array}{l}
X_{0}X_{1}+X_{2}^{2}+X_{3}X_{4}=0,\\   
X_{1}=X_{0}Y_{1}, Y_{1}X_{3}+X_{0}=0
\end{array}
\right\}
\mathrm{\ \  in\ \  } \P^{4}_{[X_{0}:\cdots:X_{4}]} \times \A^{1}_{(Y_{1})}\\
&\cong& \{Y_{1}^{3}X_{3}^{2}+X_{2}^{2}+X_{3}X_{4}=0\} \mathrm{\ \  in\ \  } \P^{2}_{[X_{2}:X_{3}:X_{4}]} \times \A^{1}_{(Y_{1})}. 
\end{eqnarray*}
Hence each $(q|_{D_{h}})$-fiber is smooth except $D_{f}|_{D_{h}}$.

Theorem \ref{thm:main1} (B) now shows that $(Q, D_{h}, D_{f})$ is a compactification of $\A^{3}$ compatible with $q$. 
Since both $D_{h}$ and $D_{f}$ differ from $E_{\Phi|_{Q}}$, the ampleness of $D_{h}+D_{f}$ follows from \cite[Lemma 2.2]{Kis}.
\end{ex}

Thirdly, suppose that $Q$ is an arbitrary quadric fibration and $D_{f}|_{D_{h}}$ is smooth.
Then $D_{h}$ is normal by Lemma \ref{lem:nnorm}.
In fact, it holds that $D_{h} \cong \F_{d}$ for some $d \in \Z_{\geq 0}$ by Theorem \ref{thm:main1} (B).
Let us mention that in \cite[\S 4.4, Lemma 2]{Mul}, it is erroneously claimed that $d=0$.
%It is because $B_{2}(W)$ is miscalculated to be two, where $W$ is the image of the small contraction $Q \to W$ of $\Sigma_{d} \subset D_{h}$. 
In Example \ref{ex:4}, we construct an example with $D_{h} \cong \F_{d}$ for each $d \in \Z_{\geq 0}$.

\begin{ex}\label{ex:4}
Let $d \in \Z_{\geq 0}$ and $P \coloneqq \F(0,1,d)$ with the $\P^{2}$-bundle structure $p \colon P \to \P^{1}$.
For $i=1, d$, let $S_{i}$ be the sub $\P^{1}$-bundle of $P$ associated with the projection $\mc{O}_{\P^{1}} \oplus \mc{O}_{\P^{1}} (1) \oplus \mc{O}_{\P^{1}} (d) \to \mc{O}_{\P^{1}} \oplus \mc{O}_{\P^{1}} (i)$ and $F$ a $p$-fiber.
Then it holds that $S_{i} \cong \F_{i}$ and $S_{i} \sim \xi_{P}-(d+1-i)F,$. Also we have:
\begin{eqnarray}
(S_{i}|_{S_{(d+1-i)}})^{2} &=& (\xi_{P}-(d+1-i)F)^{2} \cdot (\xi_{P}-iF)\\
                             &=& \xi_{P}^{3}-(2d+2-i)\xi_{P}^{2} \cdot F=-(d+1-i).\nonumber
\end{eqnarray}
Hence $S_{d}|_{S_{1}}=\Sigma_{1}$ and $S_{1}|_{S_{d}}=\Sigma_{d}$.

Now take $B \subset S_{1}$ as a smooth member of $|2(\Sigma_{1}+f_{1})|$, which is a $p$-bisection. 
Let $\psi \colon \wt P \to P$ be the blow-up along $B$.
Then $-K_{\wt P}$ is $(p \circ \psi)$-ample.
An easy computation shows that there is the elementary link with center along $B$:
\begin{equation}
\xymatrix@!C=15pt@R=15pt{
&\wt{P} \ar[dl]_-{\psi} \ar[dr]^-{\vp}
&
\\P \ar[d]_-{p}
&
&Q \ar[d]^-{q}
\\\P^{1} \ar@{=}[rr]
& 
&\P^{1}
}
\end{equation}
such that $\vp$ is the blow-up of a quadric fibration $Q$ along a $q$-section.
In fact, this is the inverse of an elementary link as in Lemma \ref{lem:q-p}. 
Since $B \cong \P^{1}$, we have $h^{1,2}(Q)=0$ by (\ref{eq:main1-3}) and (\ref{eq:main1-5}).

Let $D_{h} \coloneqq (S_{d})_{Q}$ and $D_{f}$ a singular $q$-fiber, which exists by Lemma \ref{lem:q-p} (2).
Then $2D_{h} \sim_{\P^{1}} -K_{Q}$ by Lemma \ref{lem:q-p} (1).
Since $B \cap S_{d}=\emptyset$ and $E_{\vp}=(S_{1})_{\wt P}$, it holds that $D_{h} \cong S_{d} \cong \F_{d}$. 
Theorem \ref{thm:main1} (B) now shows that $(Q, D_{h}, D_{f})$ is a compactification of $\A^{3}$ compatible with $q$. 
\end{ex}
%\begin{ex}\label{eq:3}
%Let $P \coloneqq \{Y_{0}X_{1}=Y_{1}X_{0}\} \subset \P^{4}_{[X_{0}:\cdots:X_{4}]} \times \P^{1}_{[Y_{0}:Y_{1}]}$ and $g \colon P \to \P^{4}_{[X_{0}:\cdots:X_{4}]}$ the blow-up along a line.
%Take $Q$, $D_{h}$, $D_{f} \subset P$ as follows:
%\begin{eqnarray}
%Q \coloneqq \{Z_{0}\}
%\end{eqnarray}
%\end{ex}

Finally, we give an example of compactifications of affine homology $3$-cells into quadric fibrations which gives a negative answer to Problem \ref{prob:main} (1) without the assumption on the compatibility.

\begin{ex}\label{ex:5}
Take $H, S$ and $C$ in $\Q^{3} = \{X_{0}X_{1}+X_{2}^{2}+X_{3}X_{4}=0\} \subset \P^{4}_{[X_{0}: \cdots:X_{4}]}$ as follows:
\begin{eqnarray}
H &\coloneqq&\{X_{0}X_{1}+X_{2}^{2}+X_{3}X_{4}=X_{0}=0\},\\
S &\coloneqq&\{X_{0}X_{1}+X_{2}^{2}+X_{3}X_{4}=0, X_{0}X_{3}^2=X_{4}^{3}\},\\
C &\coloneqq&\{X_{0}X_{1}+X_{2}^{2}+X_{3}X_{4}=0, X_{3}=X_{4}=X_{0}\}.
\end{eqnarray}
As in Example \ref{ex:3}, the blow-up $Q \coloneqq \Bl_{C} \Q^{3}$ has a quadric fibration structure $q \colon Q \to \P^{1}$.
Since each $q$-fiber is the strict transform of a hyperplane section of $\Q^{3}$ containing $C$, both $H_{Q}$ and $S_{Q}$ are not $q$-fibers.

Now set $U \coloneqq Q \setminus (H_{Q} \cup S_{Q})$, $\Q^{0} \coloneqq \Q^{3} \setminus H$, $S^{0} \coloneqq S \setminus (S \cap H)$ and $C^{0} \coloneqq C \setminus (C \cap  H)$.
Then $U$ is the affine modification of $\Q^{0}$ with the locus $(C^{0} \subset S^{0})$.
In $\P^{4} \setminus H \cong \A^{4}_{[x_{1}, \ldots, x_{4}]}$, we have an isomorphism $\Q^{0} \cong \{x_{1}+x_{2}^{2}+x_{3}x_{4}=0\} \cong \A^{1}_{[x_{2}]} \times \A^{2}_{[x_{3}, x_{4}]}$. 
This isomorphism sends $S^{0}$ and $C^{0}$ to 
$\A^{1}_{[x_{2}]} \times \{x_{3}^2=x_{4}^{3}\}$ and $\A^{1}_{[x_{2}]} \times \{(1,1)\}$ respectively.
By \cite{tDP}, $U$ is isomorphic to $\A^{1}_{[x_{2}]} \times V(3,2)$, where $V(3,2)=\{z^{2}x_{4}^{3}+3zx_{4}^{2}+3x_{4}-zx_{3}^{2}-2x_{3}=1\} \subset \A^{3}_{[x_{3},x_{4},z]}$ is an affine homology $2$-cell of logarithmic Kodaira dimension one.
Hence $(Q, H_{Q} \cup S_{Q})$ is a compactification of an affine homology $3$-cell $\A^{1} \times V(3,2)$.
We note that $\A^{1} \times V(3,2) \not \cong \A^{3}$ by \cite[Theorem 1]{I-F}.
\end{ex}

\section{Compactifications of affine homology $3$-cells compatible with $\P^{2}$-bundles}\label{sec:mainA}

In this section, we will give a solution of Problem \ref{prob:main} for compactifications compatible with $\P^{2}$-bundles.
Theorem \ref{thm:maindiag-A} follows as a corollary.

%\subsection{Compactification of affine homology $3$-cells compatible with $\P^{2}$-bundles}\label{sec:P2}

First we give the solution of Problem \ref{prob:main} (1) for such compactifications.

\begin{lem}
Let $p \colon P \ra C$ be a $\P^{2}$-bundle, $D_{h}$ a reduced effective divisor on $P$, and $D_{f}$ a $p$-fiber.
Then the following are equivalent.
\begin{enumerate}
\item[\textup{(1)}] The complement $P \setminus (D_{h} \cup D_{f})$ is an affine homology $3$-cell. 
\item[\textup{(2)}] $C \cong \P^{1}$ and $D_{h}$ is a sub $\P^{1}$-bundle.
\item[\textup{(3)}] It holds that $P \setminus (D_{h} \cup D_{f}) \cong \A^{3}$.
\end{enumerate}
\end{lem}

\begin{proof}
Since $(3) \Rightarrow (1)$ is trivial and $(2) \Rightarrow (3)$ follows from \cite[Lemma 5.15]{Kis}, we only have to show $(1)\Rightarrow (2)$.

Suppose that (1) holds. The first assertion follows from Lemma \ref{lem:sub-1}.
By the same argument as in the proof of Lemma \ref{lem:sub-2}, $D_{h}$ is a prime divisor such that $\Pic P = \Z D_{f} \oplus \Z D_{h}$. 
On the other hand, we have $\Pic P = \Z D_{f} \oplus \Z \xi_{P}$.
Hence $D_{h} \sim_{C} \xi_{P}$, which implies that $D_{h}$ is a sub $\P^{1}$-bundle of $P$.
\end{proof}

Next we characterize $(P', D_{h,2}, D_{f,2})$ as in Example \ref{ex:0-A}.

\begin{lem}\label{lem:001}
Let $p \colon P \ra \P^{1}$ be a $\P^{2}$-bundle with associated vector bundle $\mc{E}$.
Let $F$ be a $p$-fiber and $D \subset P$ a sub $\P^{1}$-bundle.
Suppose that $\deg \mc{E}=3n+1$, $D \cong \F_{0}$ and $D \sim \xi_{P}-(n+1)F$ for some $n \in \Z$.
Then we have $\mc{E} \cong \mc{O}_{\P^{1}}(n) \oplus \mc{O}_{\P^{1}}(n) \oplus \mc{O}_{\P^{1}}(n+1)$, 
and $D$ is the exceptional divisor of the blow-up $f \colon P \cong \F(0,0,1) \ra \P^{3}$ along a line. 
\end{lem}

\begin{proof}
By replacing $\mc{E}$ by $\mc{E} \otimes \mc{O}_{\P^{1}}(-n)$, we may assume that $n=0$.
Let us show the ampleness of $-K_{P}$. 
It is obvious that $-K_{P}|_{F}$ is ample.
By the canonical bundle formula and the adjunction formula, we have:
\begin{eqnarray}
\ \ \ \ -K_{P} &\sim& 3\xi_{P}+F \sim 3D + 4F.\label{eq:6.0.1}\\
\ \ \ \ D|_{D} &\sim& -\frac12 (K_{P}+D)|_{D}-2F|_{D} \sim -\frac12 K_{D}-2f_{0} \sim \Sigma_{0}-f_{0}.
\end{eqnarray}
%D|_{D} &\sim& \frac12 (-K_{D})-2F|_{D} \sim \Sigma_{0}-f_{0}.
we thus get $-K_{P}|_{D} \sim (3D + 4F)|_{D} \sim 3\Sigma_{0}+f_{0}$, which is also ample.
 
Suppose that $(-K_{P} \cdot r) \leq 0$ holds for some curve $r \subset P$. 
Since both $-K_{P}|_{F}$ and $-K_{P}|_{D}$ are ample, (\ref{eq:6.0.1}) now shows that $r$ must be disjoint from any $p$-fiber, a contradiction.
Hence $-K_{P}$ is strictly nef.
On the other hand, we have $(-K_{P})^{3}=54$ since $P$ is a $\P^{2}$-bundle over $\P^{1}$. Hence $-K_{P}$ is big and semiample by the base-point free theorem.
Since $-K_{P}$ is strictly nef and semiample, it is ample.

Therefore $P$ is a Fano $\P^{2}$-bundle. 
By \cite[Table 2]{M-M}, $P$ is isomorphic to either $\P^{1} \times \P^{2}$ or $\F(0,0,1)$.
Since $\deg \mc{E}=1$, it holds that $P \cong \F(0,0,1)$ and $\mc{E} \cong \mc{O}_{\P^{1}} \oplus \mc{O}_{\P^{1}} \oplus \mc{O}_{\P^{1}}(1)$, which is the first assertion.

Since $F \sim f^{*}\mc{O}_{\P^{3}}(1)-E_{f}$ and $-K_{P} \sim f^{*}\mc{O}_{\P^{3}}(4)-E_{f} \sim 3E_{f}+4F$, the second assertion follows from (\ref{eq:6.0.1}).
\end{proof}

Now we can give a solution to Problem \ref{prob:main} (2) for compactification compatible with $\P^{2}$-bundles.

\begin{prop}\label{prop:p}
Let $(P, D_{h}, D_{f})$ be a compactification of $\A^{3}$ compatible with a $\P^{2}$-bundle $p \colon P \ra \P^{1}$. 
We follow the notation of Example \ref{ex:0-A}.
Regard $p(D_{f})$ and $p'(D_{f,2})$ as $\infty \in \P^{1}$.
Then there is the composition $g_{2} \colon P \dra P'$ of elementary links with center along linear subspaces in the fibers at $\infty$ such that $D_{h,2}=(D_{h})_{P'}$. 
In particular, there exists the following diagram of rational maps preserving $P \setminus (D_{h} \cup D_{f}) \cong \A^{3}$:
\begin{equation}
\xymatrix@C=15pt@R=15pt{
(P, D_{h}, D_{f}) \ar@{.>}[r]^-{g_{2}} \ar[d]_{p}
&(P', D_{h,2}, D_{f,2}) \ar[d]_{p'} \ar[r]^-{g_{3}}
&(\P^{3}, H)
\\ \P^{1} \ar@{=}[r]
&\P^{1},
&
}
\end{equation}
where $H\coloneqq g_{3*}D_{f,2}$.
\end{prop}

\begin{proof} 
Suppose that $D_{h} \cong \F_{d}$ for some $d>0$.
Take the elementary link $P \dra P_{1}$ with center at a point $p \in D_{h} \cap D_{f}$ such that $p \not \in \Sigma_{d}$. 
This elementary link preserves $P \setminus (D_{h} \cup D_{f}) \cong \A^{3}$.
Also we have $(D_{h})_{P_{1}} \cong \F_{d-1}$ because $P \dra P_{1}$ induces an elementary transform of $D_{h}$ with center at $p$.
Taking such elementary links $d$ times, we may assume that $D_{h} \cong \F_{0}$.

Let $\mc{E}$ be an associated vector bundle of $P$.
Set $d \coloneqq \deg{\mc{E}}$ and take $e \in \Z$ such that $D_{h} \sim \xi_{P}+eD_{f}$.

Let us show that $d+e \in 2\Z$.
By the canonical bundle formula and the adjunction formula, we have:
\begin{eqnarray}
-K_{X} &\sim& 3\xi_{P}-(d-2)D_{f} \sim 3D_{h}-(d+3e-2)D_{f}.\\
-K_{D_{h}} &\sim& (2D_{h}-(d+3e-2)D_{f})|_{D_{h}} \\
                &\sim& 2(D_{h}-(e-1)D_{f})|_{D_{h}}-(d+e)D_{f}|_{D_{h}}.\nonumber
\end{eqnarray} 
This gives $d+e \in 2\Z$ because $-K_{D_{h}}\sim 2(\Sigma_{0}+f_{0})$ and $D_{f}|_{D_{h}} \sim f_{0}$.

Now let $L \subset D_{f}$ be a linear subspace and $P \dra P_{1}$ the elementary link with center along $L$.
This elementary link preserves $P \setminus (D_{h} \cup D_{f}) \cong \A^{3}$.
Let $F$ be a fiber of the induced $\P^{2}$-bundle $p_{1} \colon P_{1} \to \P^{1}$.
Take an associated vector bundle $\mc{E}'$ of $P_{1}$ as in Lemma \ref{lem:p-p}.

Consider the case where $L$ is a point outside $D_{h}$. 
Then we have $(D_{h})_{P_{1}} \cong D_{h} \cong \F_{0}$.
By Lemma \ref{lem:p-p} and Corollary \ref{cor:p-p}, we have $\deg \mc{E}'=d-1$ and $(D_{h})_{P_{1}} \sim \xi_{P_{1}}+(e+1)F$.
For each $m \in \Z_{\geq 0}$, taking such elementary links $m$ times, we can replace $(d, e)$ with $(d-m, e+m)$. 

Consider the case where $L =D_{f}\cap D_{h}$. 
Then we have $(D_{h})_{P_{1}} \cong D_{h} \cong \F_{0}$.
Replacing $\mc{E}'$ with $\mc{E}' \otimes p_{1}^{*}\mc{O}_{\P^{1}}(1)$,
we have $\deg \mc{E}'=d+1$ and $(D_{h})_{P_{1}} \sim \xi_{P_{1}}+(e-1)F$ by Lemma \ref{lem:p-p} and Corollary \ref{cor:p-p}.
For each $m \in \Z_{\geq 0}$, taking such elementary links $m$ times, we can replace $(d, e)$ with $(d+m, e-m)$. 

Now set $m \coloneqq \frac{d+3e}{2}+1 \in \Z$ 
and replace $(d, e)$ with $(d+m, e-m)=(\frac{3(d+e)}{2}+1, -\frac{d+e}{2}-1)$.
Applying Lemma \ref{lem:001} with $n = \frac{d+e}{2}$, we have the assertion.
\end{proof}

%\subsection{Construction of $g_{1}, g_{2}$ and $g_{3}$.}\label{sec:desiredseq1}

Now we can prove Theorem \ref{thm:maindiag-A}.

\begin{proof}[Proof of Theorem \ref{thm:maindiag-A}]
We have shown that $\Sing D_{h}$ is a $q$-section in Lemma \ref{lem:nnorm}.
By Lemma \ref{lem:q-p}, there is the elementary link $g_{1} \colon Q \dra P$ with center along $\Sing D_{h}$ and the induced morphism $p \colon P \to \P^{1}$ is a $\P^{2}$-bundle.
Let $E$ be the exceptional divisor of the elementary link.
As in the proof of Theorem \ref{thm:main1} (A), we can show that $g_{1}$ induces an isomorphism
$\A^{3} \cong Q \setminus (D_{h} \cup D_{f}) \cong P \setminus (E \cup (D_{f})_{P}).$
Hence $(P, D_{h,1}, D_{f,1}) \coloneqq (P, E, (D_{f})_{P})$ is a compactification of $\A^{3}$ compatible with $p$, which proves (1).
The assertions (2) follow from Proposition \ref{prop:p}.
\end{proof}

\section{Proof of Theorem \ref{thm:maindiag-B}}\label{sec:mainB}

The remainder of the paper will be devoted to the proof of Theorem \ref{thm:maindiag-B}.
From now on, 
we assume that $(Q, D_{h}, D_{f})$ is a compactification of $\A^{3}$ compatible with a quadric fibration $q \colon Q \to C \cong \P^{1}$ such that $D_{h}$ is normal. Also we use the following notation:

\begin{nota}
For $d \in \Z_{>0}$, we will denote by $S_{d}$ the blow-up of $\F_{d}$ at a point outside $\Sigma_{d}$. We note that $S_{d}$ is also the blow-up of $\F_{d-1}$ at a point in $\Sigma_{d-1}$.
\end{nota}

\subsection{Singularities of $D_{h}$ and $D_{f}|_{D_{h}}$}\label{sec:sing}

First, we establish a relation between the singularity of $D_{f}|_{D_{h}}$ and that of $D_{h}$.
Theorem \ref{thm:main1} (B) shows that $D_{f} \cong \Q^{2}_{0}$ and $D_{h}|_{D_{f}} \sim \mc{O}_{\Q^{2}_{0}}(1)$.  
Hence $D_{h}|_{D_{f}}$ is either a smooth conic, the union of two distinct rulings, or a non-reduced curve supporting on a ruling of $\Q_{0}^{2}$. 

\begin{thm}\label{thm:main2}
We have the following correspondence.
\begin{enumerate}
\item[\textup{(1)}] If $D_{f}|_{D_{h}}$ is smooth, then $D_{h} \cong \F_{d}$ for some $d \geq 0$. 
\item[\textup{(2)}] If $D_{f}|_{D_{h}}$ is reducible, then $D_{h} \cong S_{d}$ for some $d > 0$.
\item[\textup{(3)}] If $D_{f}|_{D_{h}}$ is non-reduced, then $D_{h}$ has either exactly two DuVal singularities of type $A_{1}$, or the unique DuVal singularity of type $A_{3}$ or $D_{m}$ $(m \geq 4)$.
\end{enumerate}
\end{thm}

\begin{proof}
Take a $q$-section $s \subset D_{h}$ and construct $\vp, \psi, p, \wt Q, P$ and $B$ as in Lemma \ref{lem:q-p}.
Write $G \coloneqq \psi_{*} E_\vp$,  $\infty \coloneqq q(D_{f})$, $f_{\infty} \coloneqq (p|_{G})^{*}(\infty)$ and $l_{t} \coloneqq (p|_{(D_{h})_{P}})^{*}(t)$ for $t \in C$.

Recall that $B \cong \P^{1}$, $D_{f} \cong \Q^{2}_{0}$ and singular $(q|_{D_{h}})$-fibers are at most $D_{f}|_{D_{h}}$ by Theorem \ref{thm:main1} (B).
In particular, $p|_{B}$ is ramified over $\infty$.
By Lemma \ref{lem:fiber}, $(p|_{G})$-fibers contained in $(D_{h})_{P}$ are at most $f_{\infty}$. 
Also $f_{\infty} \not\subset (D_{h})_{P}$ if and only if $D_{f}|_{D_{h}}$ is reduced.

By Lemma \ref{lem:sub-2}, we have $2D_{h} \sim_{C} -K_{Q}$.
Hence $(D_{h})_{P}$ is a sub $\P^{1}$-bundle of $P$ not containing $B$ by Lemma \ref{lem:q-p} (1). 
Since $G$ is also a sub $\P^{1}$-bundle of $P$, there exists a unique $p$-section $s'$ 
and $a \in \Z_{\geq 0}$ such that
%\begin{equation}
$(D_{h})_{P}|_{G}=s'+a f_{\infty}.$
%\end{equation}

Let us show that $(B \cdot s')_{G} \leq 1$.
For $t \in C$, it holds that $l_{t} \cap B = \emptyset$ if and only if $(q|_{D_{h}})^{*}(t)$ is smooth by Lemma \ref{lem:fiber}. 
Since singular $(q|_{D_{h}})$-fibers are at most $D_{f}|_{D_{h}}$, we have $\Supp ((D_{h})_{P} \cap {B}) \subset f_{\infty}$. 
Hence $\Supp (s' \cap B) \subset \Supp (f_{\infty} \cap B)$. 
Since $p|_{B}$ is ramified over $\infty$, the support of $f_{\infty} \cap B$ is a point.
By the same reason, $B$ and $f_{\infty}$ have the same tangent direction at $\Supp (f_{\infty} \cap B)$ in $G$.
Since $(f_{\infty} \cdot s')_{G}=1$, we have $(B \cdot s')_{G} \leq 1$ as desired.

\noindent(1): Suppose that $D_{f}|_{D_{h}}$ is smooth. 
Then we have $a=0$ and $l_{\infty} \cap B= \emptyset$ by Lemma \ref{lem:fiber}.
The former implies that $(D_{h})_{P} \cup G$ is a SNC divisor, 
and the latter implies that $B \cap (D_{h})_{P} = \emptyset$.
Hence $\psi$ is an isomorphism along $(D_{h})_{\wt Q}$. 
Since $(D_{h})_{\wt Q} \cup E_{\vp} = (D_{h})_{\wt Q} \cup G_{\wt Q}$ is a SNC divisor, we have $D_{h} \cong (D_{h})_{\wt Q} \cong (D_{h})_{P}$, which is a Hirzebruch surface, and (1) is proved.

\noindent(2): Suppose that $D_{f}|_{D_{h}}$ is reducible. 
Then we have $a=0$ and $l_{\infty} \cap B \neq \emptyset$ by Lemma \ref{lem:fiber}.
Hence $(D_{h})_{P} \cup G$ is a SNC divisor. 
Also $(D_{h})_{\wt Q}$ is the blow-up of $(D_{h})_{P}$ at a point because $((D_{h})_{P} \cdot B)_{P}=(e \cdot B)_{G}=1$.
Since $(D_{h})_{\wt Q} \cap E_{\vp}$ is the strict transform of $s'$ in $\wt Q$, the divisor $(D_{h})_{\wt Q} \cup E_{\vp}$ is a SNC divisor and hence we have $D_{h} \cong (D_{h})_{\wt Q}$, which is the blow-up of a Hirzebruch surface at a point, and (2) is proved.

\noindent(3): Suppose that $D_{f}|_{D_{h}}$ is non-reduced. 
Then we have $a \geq 1$ by Lemma \ref{lem:fiber}.
Set $m \coloneqq (B \cdot (D_{h})_{P})_{P}=(B \cdot s')_{G} +2a \geq 2$. 

For $0 \leq i \leq m-1$, we define $P_{i}$, $\wt Q_{i}$, $x_{i}$, $h_{i}$ and $\psi_{i}$ by induction as follows.
Let $P_{0} \coloneqq P$, $\wt Q_{0} \coloneqq \wt Q$, $x_{0} \coloneqq \Supp((D_{h})_{P} \cap B)$, $h_{0} \coloneqq \mathrm{id}_{P}$ and $\psi_{0} \coloneqq \psi$.
For $i>0$, denote by $h_{i} \colon P_{i} \rightarrow P_{i-1}$ the blow-up at $x_{i-1}$. 
Let $x_{i} \coloneqq \Supp ((D_{h})_{P_{i}} \cap B_{P_{i}})$, which is a point.
We also define $\psi_{i} \colon \wt Q_{i} \ra P_{i}$ as the blow-up along $B_{P_{i}}$.

Then we have the following diagram by Lemma \ref{lem:blnorm} (1), 
where $\vp_{i} \colon \wt Q_{i} \ra \wt Q_{i-1}$ is the blow-up along $(\psi_{i-1})^{-1}(x_{i-1})$ for $1 \leq i \leq m-1$.
\begin{equation}
\xymatrix{
\wt Q_{m-1} \ar[r]^-{\vp_{m-1}} \ar[d]^-{\psi_{m-1}}
&\wt Q_{m-2} \ar[r]^-{\vp_{m-2}} \ar[d]^-{\psi_{m-2}}
&\cdots  \ar[r]^-{\vp_{2}} 
&\wt Q_{1} \ar[r]^-{\vp_{1}} \ar[d]^-{\psi_{1}} 
&\wt Q_{0} \ar[r]^-{\vp} \ar[d]^-{\psi_{0}=\psi} 
&Q
\\ P_{m-1} \ar[r]_-{h_{m-1}}
& P_{m-2} \ar[r]_-{h_{m-2}}
& \cdots \ar[r]_-{h_{2}}
& P_{1} \ar[r]_-{h_{1}}
& P_{0}.
&
}
\end{equation}

Let $\alpha \colon (D_{h})_{\wt Q_{m-1}} \ra D_{h}$ be the induced morphism.
To know the singularities on $D_{h}$, it suffices to detect that of $(D_{h})_{\wt Q_{m-1}}$ and the shape of \nolinebreak$E_{\alpha}$.

For $1 \leq i \leq m-1$, it holds that $(D_{h})_{P_{i}}$ is smooth and $((D_{h})_{P_{i}} \cdot B_{P_{i}})_{P_{i}}=m-i$ because $h_{i}$ is the blow-up at the point $x_{i-1}$.
Hence $(D_{h})_{P_{m-1}}$ intersects with $B_{P_{m-1}}$ at $x_{m-1}$ transversally
and $(D_{h})_{\wt Q_{m-1}}$ is the blow-up of $(D_{h})_{P_{m-1}}$ at $x_{m-1}$, which is also smooth.

Let us reveal the precise location of $x_{i} \in (D_{h})_{P_{i}}$ for $0 \leq i \leq m-1$ to detect the shape of $E_{\alpha}$.
Note that $x_{i} \in E_{h_{i}}$ by construction.
We already showed that $x_{0} = \Supp (B \cap f_{\infty})$.
Since $(B \cdot f_{\infty})_{G}=2$, we have $(B_{P_{1}} \cdot (f_{\infty})_{P_{1}})_{G_{P_{1}}}=1$.
Hence we have $x_{1}=\Supp(B_{P_{1}} \cap (f_{\infty})_{P_{1}})=\Supp (E_{h_{1}} \cap (f_{\infty})_{P_{1}})$.

We now turn to the case $i \geq 2$.
We have $x_{2} \not\in (f_{\infty})_{P_{2}}$ since $(B_{P_{2}} \cdot (f_{\infty})_{P_{2}})_{G_{P_{2}}}=0$.
Also we have $x_{2} \not \in (E_{h_{1}})_{P_{2}}$ since $(B_{P_{2}} \cdot (E_{h_{1}})_{P_{2}})_{P_{2}}=0$.
Hence $x_{2} \in E_{h_{2}}$ and $x_{2} \not\in (E_{h_{1}})_{P_{2}} \cup (f_{\infty})_{P_{2}}$.
Similarly, for $i \geq 3$, we have $x_{i} \in E_{h_{i}}$ and $x_{i} \not\in (E_{h_{i-1}})_{P_{i}}$.

%Now we can detect the shape of $E_{\alpha}$.
Let $e_{i}$ be the strict transform of $E_{h_{i}}|_{D_{h,i}}$ in $\wt Q_{m-1}$ for $1 \leq i \leq m-1$.
Set $\wt f_{\infty} \coloneqq (f_{\infty})_{\wt Q_{m-1}}$, $\wt s \coloneqq s'_{\wt Q_{m-1}}$ and $r \coloneqq E_{\psi_{m-1}}|_{(D_{h})_{\wt Q_{m-1}}}$.
By the above observation on $x_{i}$, the configuration of $e_{i}$, $\wt f_{\infty}$, $\wt s$ and $r$ in $(D_{h})_{\wt Q_{m-1}}$ is as in FIGURE \ref{fig:1}.

\begin{figure}[htbp]
  \begin{center}
    \begin{tabular}{c}

      % 1
      \begin{minipage}{0.22\hsize}
        \begin{center}
          \includegraphics[clip, width=30mm, bb=270 0 600 840]{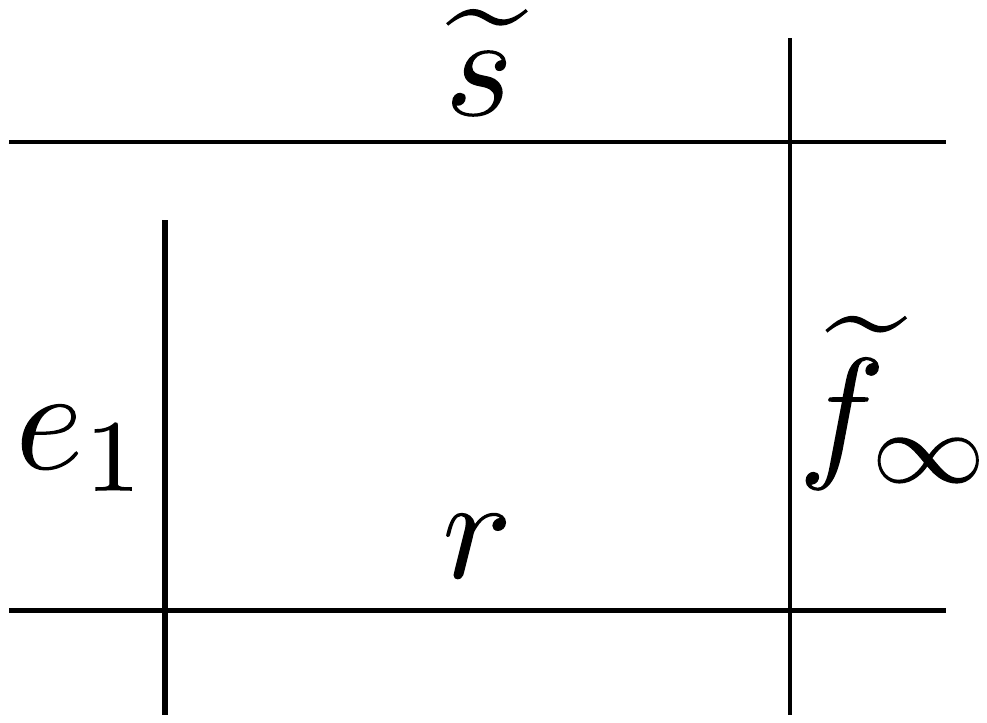}
          %\vspace{48mm}\\
                   %\hspace{-5mm}
           Case $m=2$
          \end{center}
          \vspace{2.5mm}
      \end{minipage}

      % 2
      \begin{minipage}{0.22\hsize}
        \begin{center}
          \includegraphics[clip, width=30mm, bb=270 0 600 840]{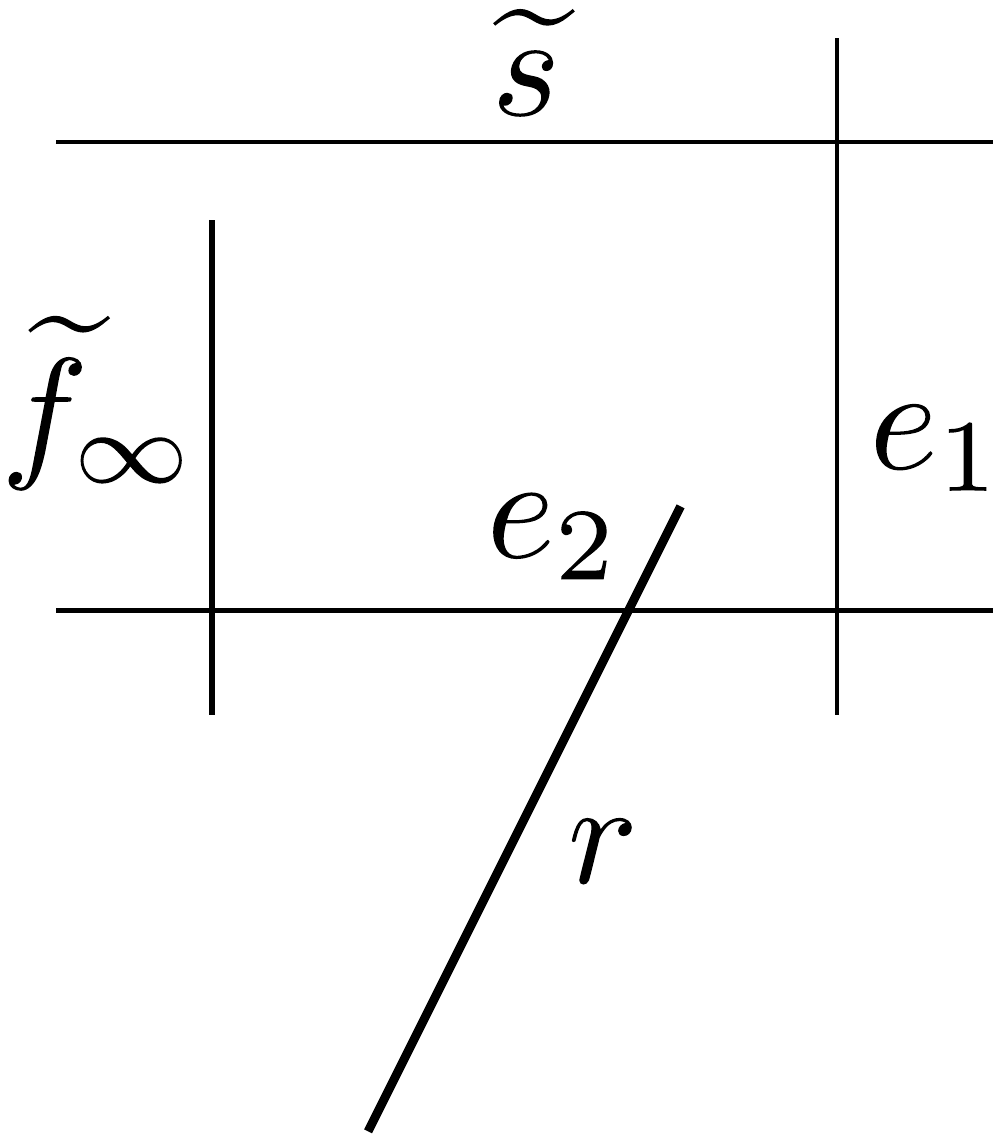}
          %\vspace{38mm}\\
          %\hspace{-5mm} 
          Case $m=3$
        \end{center}
        \vspace{2.5mm}
      \end{minipage}
      
      % 3
      \begin{minipage}{0.22\hsize}
        \begin{center}
          \includegraphics[clip, width=30mm, bb=270 0 600 840]{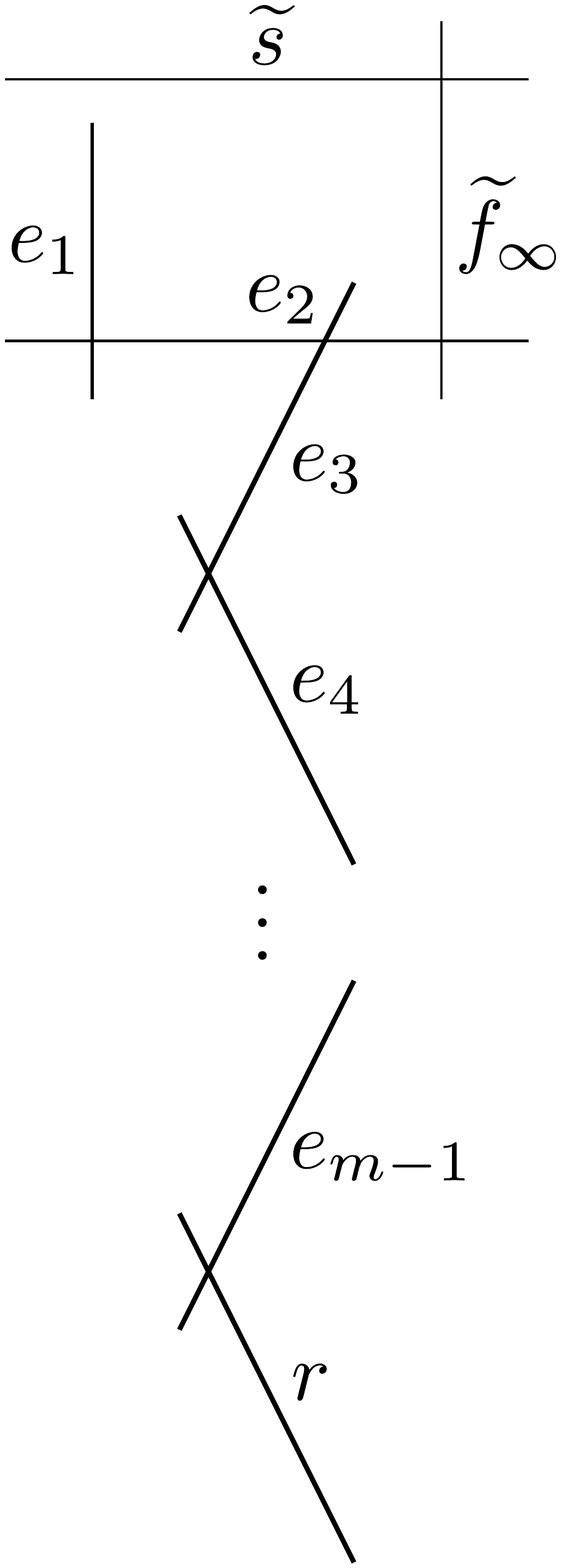}
          \hspace{1.6cm} Case $m \geq 4$ and $m$ is even
        \end{center}
      \end{minipage}
      
      % 4
      \begin{minipage}{0.22\hsize}
        \begin{center}
          \includegraphics[clip, width=30mm, bb=270 0 600 840]{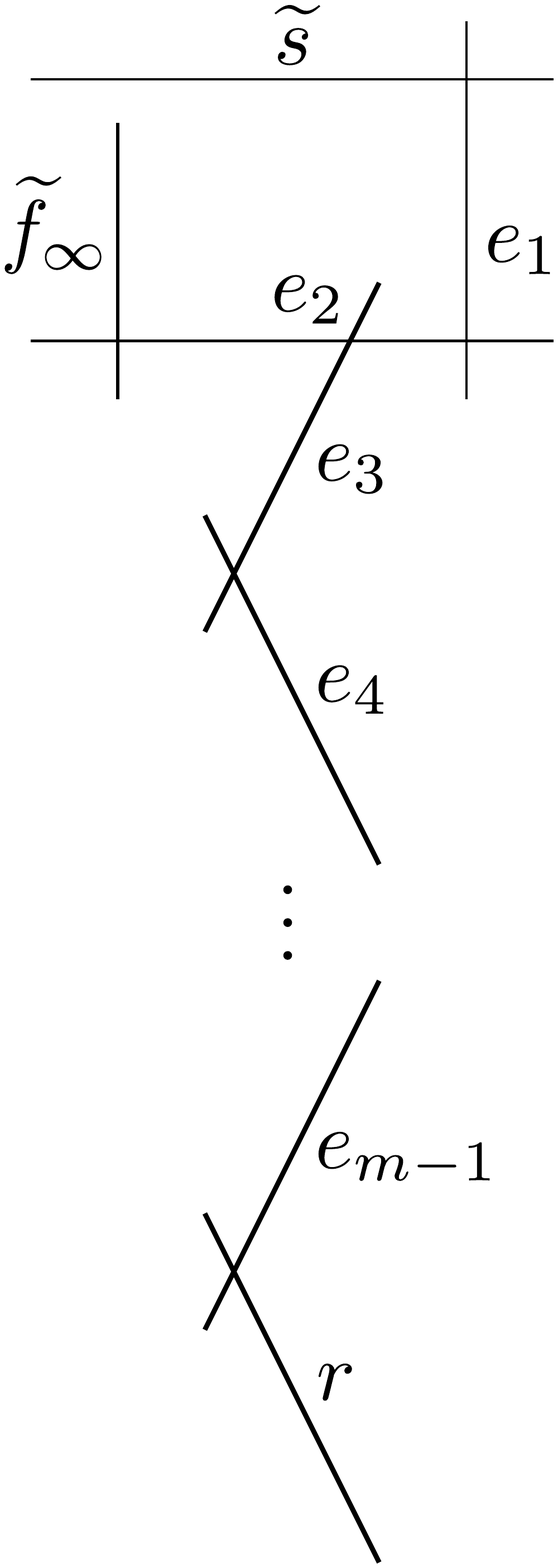}
          \hspace{1.6cm} Case $m \geq 4$ and $m$ is odd
        \end{center}
      \end{minipage}

    \end{tabular}
    \caption{The configuration of $e_{i}$, $\wt f_{\infty}$, $\wt s$ and $r$ in $(D_{h})_{\wt Q_{m-1}}$\label{fig:1}}
    \label{fig:lena}
  \end{center}
\end{figure}

It is clear that $(f_{\infty})_{\wt {Q}}$ is the exceptional divisor of $(D_{h})_{\wt Q} \to D_{h}$.
On the other hand, by Lemma \ref{lem:blnorm} (1), $e_{m-1}$ is the exceptional divisor of $(D_{h})_{\wt Q_{m-1}} \to (D_{h})_{\wt Q_{m-2}}$.  
Repeated application of Lemma \ref{lem:blnorm} (1) shows us that 
$E_{\alpha} = \wt f_{\infty} \cup \bigcup_{i=1}^{m-1}e_{i}$. Since each of them is $(-2)$-curve in $(D_{h})_{\wt Q_{m-1}}$, the singularity of $D_{h}$ is the DuVal singularity of type $2A_{1}$ when $m=2$, $A_{3}$ when $m=3$ and $D_{m}$ when $m \geq 4$, which completes the proof.
\end{proof}

The next aim is to construct explicit birational maps preserving $\A^{3}$ from $(Q, D_{h}, D_{f})$ to an another compactification $(Q', D_{h}', D_{f}')$ compatible with quadric fibration such that the singularity of $D_{h}'$ is milder than that of $D_{h}$.
To do so, we define the type of $(Q, D_{h}, D_{f})$ as follows.
%in Definition \ref{defi:main}.

\begin{defi}\label{defi:main}
Let $m \in \Z_{\geq 0}$.
We call $(Q, D_{h}, D_{f})$ \textit{a compactification of}
\begin{itemize}
\item \textit{type} $0$ when $D_{h} \cong \F_{d}$ for some $d \geq 0$. 
\item \textit{type} $1$ when $D_{h} \cong S_{d}$ for some $d > 0$.
\item \textit{type} $2$ when $D_{h}$ has two DuVal singularities of type $A_{1}$.
\item \textit{type} $3$ when $D_{h}$ has a DuVal singularity of type $A_{3}$
\item \textit{type} $m (\geq 4)$ when $D_{h}$ has a DuVal singularity of type $D_{m}$.
\end{itemize}
\end{defi}

We note that $D_{f}|_{D_{h}}$ contains a ruling of $\Q^{2}_{0} \cong D_{f}$ if and only if $m>0$.
It is easy to check that the number of the type coincides with $(B \cdot (D_{h})_{P})_{P}$ as in the proof of Theorem \ref{thm:main2}.
Hence we have the following.

\begin{cor}\label{cor:type} 
Take any $q$-section $s \subset D_{h}$ and construct $P$ and $B$ as in Lemma \ref{lem:q-p}.
Then $(Q, D_{h}, D_{f})$ is of type $m \in \Z_{\geq 0}$ if and only if $(B \cdot (D_{h})_{P})_{P}=m$.
\end{cor}

\subsection{The case of singular $D_{f}|_{D_{h}}$}\label{sec:decrease}

Next we shall give an elementary link from each compactification of $\A^{3}$ of type $m>0$ to that of type $(m-1)$. 
Composing such elementary links, we get a birational map from each compactification of $\A^{3}$ of type $m>0$ to that of type $0$.

\begin{lem}\label{lem:pqqpp}
Let $q \colon Q \ra C$ be a quadric fibration, $F$ a singular $q$-fiber, $s \subset Q$ a $q$-section and $l$ the ruling of $F \cong \Q^{2}_{0}$ which intersects with $s$.
We use the same letter $l$ and $s$ for their strict transformations by abuse of notation.
Consider the following four elementary links:
\begin{itemize}
%Firstly, let 
\item $Q \xla{\vp_{1,1}} Q_{1,1} \xra{\psi_{1,1}} Q'$: the elementary link with center along $l$.
%After that, we take 
\item $Q' \xla{\vp_{1,2}} Q_{1,2} \xra{\psi_{1,2}} P_{1}$: the elementary link with center along $s$.
%
%On the other hand, let 
\item $Q \xla{\vp_{2,1}} Q_{2,1} \xra{\psi_{2,1}} P$: the elementary link with center along $s$.
%After that, we take 
\item $P \xla{\vp_{2,2}} Q_{2,2} \xra{\psi_{2,2}} P_{2}$: the elementary link with center at the point $x \coloneqq \psi_{2,1}(l)$.
\end{itemize}

Summarizing these notation, we have the following diagram:
\begin{equation}\label{diag:pqqpp}
\xymatrix@!C=17pt@R=15pt{
&{Q}_{1,2} \ar[dl]_-{\psi_{1,2}} \ar[dr]^-{\vp_{1,2}}
&
&Q_{1,1} \ar[dl]_-{\psi_{1,1}} \ar[dr]^-{\vp_{1,1}}
&
&{Q}_{2,1} \ar[dl]_-{\vp_{2,1}} \ar[dr]^-{\psi_{2,1}}
&
&{Q}_{2,2} \ar[dl]_-{\vp_{2,2}} \ar[dr]^-{\psi_{2,2}}
&
\\P_{1} \ar[d]
&
&Q' \ar[d]
&
&Q \ar[d]
&
&P \ar[d]
&
&P_{2} \ar[d]
\\ C \ar@{=}[rr]
& 
&C \ar@{=}[rr]
& 
&C \ar@{=}[rr]
& 
&C \ar@{=}[rr]
& 
&C,
}
\end{equation}
Then the birational map $\iota \colon P_{1} \dra P_{2}$ induced by $(\ref{diag:pqqpp})$ is an isomorphism.
\end{lem}

\begin{proof}
By \cite[Proposition 3.5]{Cor}, we only have to show that $\iota$ is an isomorphism in codimension one.

Let $l'$ be the strict transform of the center of $\psi_{1,1}$ in $Q_{1,2}$. Let $\chi_{1} \colon X_{1} \to Q_{1,2}$ be the blow-up along $l'$. 
Since $s \subset Q_{1,1}$ is disjoint from $F_{Q_{1,1}}=E_{\psi_{1,1}}$, 
we have $X_{1} \cong \Bl_{s} Q_{1,1}$.
On the other hand, let $B$ be the strict transform of the center of $\psi_{2,1}$ in $Q_{2,2}$. Let $\chi_{2} \colon X_{2} \to Q_{2,2}$ be the blow-up along $B$. 
By Lemma \ref{lem:blnorm} (1), we have $X_{2} \cong \Bl_{l} Q_{2,1}$.
Summarizing these arguments, we have the following diagram:
\begin{equation}
\xymatrix@!C=17pt@R=15pt{
& 
&X_{1} \ar[dl]_-{\chi_{1}} \ar[dr]^-{\Bl_{s}}
&
&
&
&X_{2} \ar[dl]_-{\Bl_{l}} \ar[dr]^-{\chi_{2}}
&
&
\\
&{Q}_{1,2} \ar[dl]_-{\psi_{1,2}} \ar[dr]^-{\vp_{1,2}}
&
&{Q}_{1,1} \ar[dl]_-{\psi_{1,1}} \ar[dr]^-{\vp_{1,1}}
&
&{Q}_{2,1} \ar[dl]_-{\vp_{2,1}} \ar[dr]^-{\psi_{2,1}}
&
&{Q}_{2,2} \ar[dl]_-{\vp_{2,2}} \ar[dr]^-{\psi_{2,2}}
&
\\P_{1}
&
&Q' 
&
&Q
&
&P 
&
&P_{2},
}
\end{equation}

In $Q$, the curve $s$ intersects with $l$ transversally. Hence the induced map $X_{1} \dra X_{2}$ is the Atiyah flop.
By construction both $E_{\chi_{1}}$ and $E_{\psi_{2,2}}$ are the strict transforms of $F$.
By Lemma \ref{lem:nnorm} (2) both $E_{\chi_{2}}$ and $E_{\psi_{1,2}}$ are the strict transforms of $E_{\psi_{2,1}}$.
Therefore $\iota$ is also an isomorphism in codimension one, which completes the proof.
\end{proof}

\begin{thm}\label{thm:main4}
Suppose that $(Q, D_{h}, D_{f})$ is of type $m>0$. 
Let $l$ be an irreducible component of $\Supp (D_{f}|_{D_{h}})$ and take the elementary link $Q \la Q_{1,1} \ra Q'$ with center along $l$. 
Let $E$ be the exceptional divisor of the elementary link. 
Then $(Q', (D_{h})_{Q'}, E)$ is a compactification of $\A^{3}$ compatible with a quadric fibration of type $(m-1)$.
\end{thm}

\begin{proof}
By Lemma \ref{lem:q-q}, we have $Q \setminus ((D_{h})_{Q'} \cup E) \cong \A^{3}$. 
By Lemma \ref{lem:nnorm}, $(D_{h})_{Q'}$ is normal. 
Hence it suffices to show that $(Q', (D_{h})_{Q'}, E)$ is of type $(m-1)$.

By Theorem \ref{thm:main2} we can take a $q$-section $s \subset D_{h}$ intersecting with $l$.
Take elementary transformations as in Lemma \ref{lem:pqqpp}.
Let $B \subset P_{1}$ be the center of $\psi_{1,2}$.
By Corollary \ref{cor:type}, it suffices to show that $((D_{h})_{P_{1}} \cdot B)_{P_{1}}=m-1$.

By Lemma \ref{lem:pqqpp} we have $P_{1} \cong P_{2}$ and $B_{P}$ is the center of $\psi_{2,1}$.
Since $D_{h}|_{D_{f}}$ is not smooth, Lemma \ref{lem:fiber} now implies $x \in (D_{h})_{P} \cap B_{P}$.
Hence $(D_{h})_{Q_{2,2}} \sim \vp_{2,2}^{*}(D_{h})_{P}-E_{\vp_{2,2}} \sim \psi_{2,2}^{*}(D_{h})_{P_{1}}$ by Corollary \ref{cor:p-p} and 
$((D_{h})_{P_{1}} \cdot B)_{P_{1}}=((D_{h})_{P} \cdot B_{P})_{P}-(E_{\vp_{2,2}} \cdot B_{Q_{2,2}})_{Q_{2,2}}=m-1$ by Corollary \ref{cor:type}.
\end{proof}

An easy computation shows the following.

\begin{cor}\label{cor:type1}
We follow the notation of Theorem \ref{thm:main4}. 
Suppose that $m=1$ and take $d > 0$ such that $D_{h} \cong S_{d}$.  
Then $(D_{h})_{Q'} \cong \F_{d}$ (resp.\ $\F_{d-1}$) when $l$ intersects with (resp.\ is disjoint from) the strict transform of $\Sigma_{d}$ in $S_{d}$.
\end{cor}

\subsection{The case of smooth $D_{f}|_{D_{h}}$}\label{sec:desiredseq2}

By Theorem \ref{thm:main4}, we are reduced to prove Theorem \ref{thm:maindiag-B} for the case where $(Q, D_{h}, D_{f})$ is of type $0$, i.e. where $D_{h}$ is a Hirzebruch surface.
First we construct a birational map which decreases the degree of $D_{h}$ as a Hirzebruch surface.

\begin{lem}\label{lem:type0}
Suppose that $D_{h} \cong \F_{d}$ for some $d>0$.
Set $\infty \coloneqq q(D_{h})$.
Then there are an another compactification $(Q', D'_{h}, D'_{f})$ of $\A^{3}$ compatible with a quadric fibration $q'$
and the composition $h \colon Q \dra Q'$ of elementary links with center along rulings in the fibers at $\infty$
such that $(D_{h})_{Q'}=D'_{h} \cong \F_{d-1}$ and $D'_{f}=(q')^{*}(\infty)$. 
In particular, $h$ preserves $Q \setminus (D_{f} \cup D_{h}) \cong \A^{3}$.

\begin{equation}
\xymatrix@C=15pt@R=15pt{
(Q, D_{h}, D_{f}) \ar@{.>}[r]^-{h} \ar[d]_{q}
&(Q', D'_{h}, D'_{f})\ar[d]_{q'}
\\ \P^{1} \ar@{=}[r]
&\P^{1},
}
\end{equation}
\end{lem}

\begin{proof}
Take the elementary link $f_{1} \colon Q \dra Q_{1}$ with center along a ruling of $D_{f} \cong \Q_{0}$ which is disjoint from $\Sigma_{d} \subset D_{h}$. 
Let $E$ be the exceptional divisor of the elementary link. 
Since $f_{1}$ induces the elementary transformation of $D_{h}$ with center a point outside $\Sigma_{d}$,
we get a compactification $(Q_{1}, (D_{h})_{Q_{1}}, E)$ of $\A^{3}$ such that $(D_{h})_{Q_{1}} \cong S_{d}$.

Now take the elementary link $f_{2} \colon Q_{1} \dra Q'$ with center along the irreducible component of $\Supp (E|_{(D_{h})_{Q_{1}}})$ which is disjoint from the strict transform of $\Sigma_{d}$ in $(D_{h})_{Q_{1}}$.
Then by Corollary \ref{cor:type1}, we get a compactification $(Q', D'_{h}, D'_{f})$ of $\A^{3}$ such that $D'_{h} \cong \F_{d-1}$. 
Hence $h \coloneqq f_{2} \circ f_{1}$ is the desired birational map.
\end{proof}

Repeated application of Lemma \ref{lem:type0} enables us to assume that $(Q, D_{h}, D_{f})$ satisfies $D_{h} \cong \F_{0}$.
Next we show that such a compactification is the same as $(Q', D_{h}', D_{f}')$ as in Example \ref{ex:0-B}.

\begin{lem}\label{lem:qstandard}
Suppose that $D_{h} \cong \F_{0}$.
Then $Q$ is the blow-up of $\Q^{3}$ along a smooth conic and $D_{h}$ is the exceptional divisor of the blow-up.
\end{lem}

\begin{proof}
First let us show the ampleness of $-K_{Q}$.
By Lemma \ref{lem:sub-2}, we can take $a \in \Z$ such that $-K_{Q} \sim 2D_{h}+aD_{f}$.
By the adjunction formula, we have 
$D_{h}|_{D_{h}} \sim -K_{D_{h}}-aD_{f}|_{D_{h}}$.
Since $D_{h} \cong \F_{0}$, we have:
\begin{eqnarray}\label{eq:main3-1}
D_{h}^{3}
&=&(K_{D_{h}}+aD_{f}|_{D_{h}})^{2}\\
&=&(K_{D_{h}})^{2}+2a(K_{D_{h}} \cdot D_{f}|_{D_{h}})=8-4a.\nonumber\\
%Combining this and Lemma \ref{lem:q-p} (3), we have
(-K_{Q})^3
&=&(2D_{h} + a D_{f})^3\label{eq:main3-2}\\
&=&8D_{h}^3+12a(D_{h}^2 \cdot D_{f})=64-8a.\nonumber
\end{eqnarray}

On the other hand, by Lemma \ref{lem:q-p} (3), it holds that $(-K_{Q})^{3}=40-(8p_{a}(B)+32p_{a}(C))$.
We have $C \cong \P^{1}$ by assumption. 
Combining Theorem \ref{thm:main1} (B), (\ref{eq:main1-3}) and (\ref{eq:main1-5}), we get $p_{a}(B)=0$.
Hence $(-K_{Q})^3=40$. 
Substituting this into (\ref{eq:main3-2}), we have $a=3$. 
Hence we have
\begin{equation}
-K_{Q} \sim 2D_{h}+3D_{f} \label{eq:7.3.4}
\end{equation}
and $-K_{Q}|_{D_{h}} \sim (2D_{h}+3D_{f})_{D_{h}} \sim -2K_{D_{h}}-3D_{f}|_{D_{h}} \sim 4\Sigma_{0}+f_{0}$, which is ample. 
Clearly $-K_{Q}|_{D_{f}}$ is also ample.

Suppose that $(-K_{Q} \cdot r) \leq 0$ holds for some curve $r \subset Q$.
Since both $-K_{Q}|_{D_{h}}$ and $-K_{Q}|_{D_{f}}$ are ample, 
(\ref{eq:7.3.4}) now shows that $r$ must be disjoint from any $q$-fiber, a contradiction.
Hence $-K_{Q}$ is strictly nef.
Also $-K_{Q}$ is big since $(-K_{Q})^{3}=40$ and is semiample by the base-point free theorem.
Since $-K_{Q}$ is strictly nef and semiample, it is ample.

Therefore $Q$ is a Fano quadric fibration with $(-K_{Q})^3=40$.
By \cite[Table 2]{M-M}, $Q$ is the blow-up of $\Q^{3}$ along a smooth conic, which is the first assertion.

Let $h_{2} \colon Q \to \Q^{3}$ be the blow-up morphism. 
Since $D_{f} \sim h_{2}^{*} \mc{O}_{\Q^{3}}(1) -E_{h_{2}}$ 
and $-K_{Q} \sim h_{2}^{*} \mc{O}_{\Q^{3}}(3)-E_{h_{2}} \sim 2E_{h_{2}} + 3D_{f}$,
the second assertion follows from (\ref{eq:7.3.4}).
\end{proof}

Now we can prove Theorem \ref{thm:maindiag-B}.

\begin{proof}[Proof of Theorem \ref{thm:maindiag-B}]
Suppose that $(Q, D_{h}, D_{f})$ is a compactification of $\A^{3}$ of type $m$.
Taking elementary links $m$ times as in Theorem \ref{thm:main4}, we may assume that $m=0$.
Repeated application of Lemma \ref{lem:type0} enables us to assume that $D_{h} \cong \F_{0}$.
Then $h_{1} \coloneqq \mathrm{id}_{Q}$ and $(Q', D'_{h}, D'_{f}) \coloneqq (Q, D_{h}, D_{f})$ satisfies all the assertion by Lemma \ref{lem:qstandard}. 
\end{proof}

\section*{Acknowledgement}
The author is greatly indebted to Professor Hiromichi Takagi, his supervisor, for his encouragement, comments, and suggestions. 
He wishes to express his gratitude to Professor Takashi Kishimoto and Adrien Dubouloz for his helpful comments and suggestions. 
He also would like to express his gratitude to Doctor Takeru Fukuoka for several helpful discussions concerning the technique of elementary links.

This work was supported by JSPS KAKENHI Grant Number JP19J14397 and 
the Program for Leading Graduate Schools, MEXT, Japan.

\end{document}